\documentclass{amsart}
\usepackage{mathrsfs}
\usepackage{amssymb}
\usepackage{amsmath}
\usepackage{amsthm}
\usepackage{amscd}
\usepackage{epstopdf}
\usepackage{stmaryrd}
\usepackage{stmaryrd}
\usepackage{enumerate}
\usepackage{hyperref}

\usepackage{epsfig}

\newtheorem{theorem}{Theorem}[section]
\newtheorem{conjecture}[theorem]{Conjecture}
\newtheorem*{theorem*}{Conjecture}
\newtheorem*{claim*}{Claim}
\newtheorem{claim}[theorem]{Claim}
\newtheorem{lemma}[theorem]{Lemma}
\newtheorem{proposition}[theorem]{Proposition}
\newtheorem{corollary}[theorem]{Corollary}

\newtheorem{Theorem}{Theorem}[section]

\newtheorem{Corollary}[Theorem]{Corollary}

\newtheorem{innercustomthm}{Theorem}
\newenvironment{thm}[1]
  {\renewcommand\theinnercustomthm{#1}\innercustomthm}
  {\endinnercustomthm}

\theoremstyle{remark}
\newtheorem*{remark}{Remark}

\theoremstyle{definition}
\newtheorem{definition}{Definition}

\newcommand{\Q}{\mathbb{Q}}
\newcommand{\Z}{\mathbb{Z}}
\newcommand{\R}{\mathbb{R}}
\newcommand{\C}{\mathbb{C}}
\newcommand{\F}{\mathbb{F}}
\newcommand{\bP}{\mathbb{P}}
\newcommand{\fp}{\mathfrak{p}}
\newcommand{\be}{\bf{e}}

\newcommand{\Oo}{\mathcal{O}}

\newcommand{\Mgn}{\mathcal{M}_{g,n}}

\renewcommand{\P}{\mathbb{P}}

\DeclareMathOperator{\Der}{Der}
\DeclareMathOperator{\Spec}{Spec}
\DeclareMathOperator{\GL}{GL}
\DeclareMathOperator{\hol}{hol}

\DeclareMathOperator{\tr}{tr}
\DeclareMathOperator{\SL}{SL}

\DeclareMathOperator{\Hom}{Hom}
\DeclareMathOperator{\Aut}{Aut}

\DeclareMathOperator{\PGL}{PGL}

\newcommand{\Hb}{\mathbb{H}}
\newcommand{\A}{\mathbb{A}}
\DeclareMathOperator{\SU}{SU}
\DeclareMathOperator{\SO}{SO}
\DeclareMathOperator{\Gal}{Gal}

\begin{document}
\title[Rank two $p$-curvature conjecture on generic curves]{The rank two $p$-curvature conjecture\\on generic curves}
\author{Anand Patel, Ananth N. Shankar, and Junho Peter Whang}
\date{\today}

\begin{abstract}
    We prove the $p$-curvature conjecture for rank two vector bundles with connection on generic curves, by combining deformation techniques for families of varieties and topological arguments.
\end{abstract}

\maketitle

\tableofcontents

\section{Introduction}\label{sect:1}
\subsection{\unskip}\label{sect:1.1}
This paper proves the Grothendieck--Katz $p$-curvature conjecture for rank two vector bundles with connection on generic curves punctured at generic sets of points. Formulated as follows, the conjecture gives an arithmetic criterion for certain differential equations on algebraic varieties to admit algebraic solutions. Let $(V,\nabla)$ be a vector bundle with integrable connection on a smooth variety $X$ over $\C$. Let $R \subset \C$ be a subring finitely generated over $\Z$ to which $X$ and $(V,\nabla)$ descend. We will again denote by $X$ and $(V,\nabla)$ the corresponding objects over $R$. For all but finitely many prime numbers $p$ one can reduce $X$ and $(V,\nabla)$ modulo $p$ to obtain a vector bundle with connection $(V/p,\nabla/p)$ on the $R/p$-scheme $X/p$. Associated to such a system is an operator $\psi_{p}$, its $p$-curvature, whose vanishing is equivalent to the vector bundle $V/p$ being spanned by its sheaf of parallel sections.

\begin{conjecture}[The $p$-curvature conjecture]\label{ConjA}
If $(V,\nabla)$ has vanishing $p$-curvature for almost all prime numbers $p$, then it has a full set of algebraic solutions, i.e.~it trivializes on a finite \'etale cover of $X$.
\end{conjecture}

We define a smooth curve $C/\C$ of genus $g$ with $n$ punctures to be \emph{generic} if it descends to a base $R\subset\C$ whose associated morphism $\Spec R\to\mathcal{M}_{g,n}$ to the moduli stack of curves is dominant. Our main result is:

\begin{Theorem}
\label{curve}
The $p$-curvature conjecture is true for rank $2$ vector bundles with connection on a generic curve of genus $g\geq0$ with $n\geq0$ punctures.
\end{Theorem}

Theorem \ref{curve} implies in particular that the $p$-curvature conjecture holds in rank two for curves over a Baire-generic set of complex points $\mathcal{M}_{g,n}(\C)$ on the moduli space. Other known cases of the $p$-curvature conjecture include Gauss-Manin connections by Katz \cite{katz2}, connections on certain locally symmetric varieties by Farb--Kisin \cite{fk}, and connections with solvable monodromy by Andr\'e \cite{Andre}, Bost \cite{Bost}, and  D.~Chudnovsky--G.~Chudnovsky \cite{Chudnovskis}. We will focus on the case $3g+n-3>0$ of Theorem \ref{curve} in this paper, since the case $(g,n)=(0,3)$ is known by the work of Katz on hypergeometric differential equations and the remaining cases where $2g+n\leq2$ have abelian (and therefore solvable) monodromy. 

\subsection{Local systems on topological surfaces}
\label{sect:1.2}
In Section \ref{sect:2}, we obtain conditions for finiteness of rank two representations for fundamental groups of topological surfaces. In \cite{ananth}, the second author used nodal degenerations of curves to prove that, over a generic curve $C$, every vector bundle with connection $(V,\nabla)$ with almost all $p$-curvatures zero must have finite monodromy along all simple loops on $C$. Combining this with the following topological result, we obtain a proof of Theorem \ref{curve} for genus $g\geq1$.

\begin{Theorem}
\label{top}
Let $\Sigma$ be a topological surface of genus $g\geq1$ with $n\geq0$ punctures. If a semisimple representation $\rho:\pi_1\Sigma\to\GL_2(\C)$ of its fundamental group has finite monodromy along every simple loop on $\Sigma$, then the image of $\rho$ is finite.
\end{Theorem}

The hypothesis $g\geq1$ is essential in Theorem \ref{top}, as Fuchsian triangle groups furnish counterexamples to the analogous statement for $g=0$. For surfaces of genus zero, we instead have the following. By a pair of pants on a given surface we shall mean a subsurface of genus zero with three boundary curves.

\begin{Theorem}
\label{top0}
Let $\Sigma$ be a topological surface of genus zero with $n\geq0$ punctures. If a semisimple representation $\rho:\pi_1\Sigma\to\GL_2(\C)$ has finite restriction to every pair of pants on $\Sigma$, then the image of $\rho$ is finite.
\end{Theorem}

Given a rank $2$ bundle with connection on a generic curve of genus zero almost all of whose $p$-curvatures vanish, we shall prove (Theorem \ref{genuszero}) that its monodromy representation (if semisimple) satisfies the conditions of Theorem \ref{top0}; this allows us to conclude Theorem \ref{curve} in the genus zero case. Our strategy behind the proof of Theorems \ref{top} and \ref{top0} is the following. Upon reducing to the case of representations into $\SL_2(\C)$, we shall show that, if a representation $\rho:\pi_1\Sigma\to\SL_2(\C)$ satisfies the hypotheses of Theorem \ref{top} or \ref{top0}, then
\begin{enumerate}
	\item (nonarchimedean bound) $\tr\rho(a)$ is an algebraic integer for all $a\in\pi_1\Sigma$; and
	\item (archimedean bound) the image of $\rho$ is conjugate to a subgroup of $\SU(2)$.
\end{enumerate}
By (1), we can assume without loss of generality that $\rho$ has image in $\SL_2(\bar\Q)$. Moreover, any conjugate of $\rho$ by an element $\sigma\in\Gal(\bar\Q/\Q)$ of the absolute Galois group of $\Q$ also satisfies the hypotheses of Theorem \ref{top} or \ref{top0}, since having finite order is invariant under $\Gal(\bar\Q/\Q)$-action. Thus, the eigenvalues of $\rho(a)$ are algebraic integers whose Galois conjugates all have absolute value $1$ in $\C$, by (1) and (2). By Kronecker's theorem, these eigenvalues must be roots of unity, and $\rho(a)$ has finite order for every $a\in\pi_1\Sigma$ (note that $\rho(a)$ is semisimple by (2)). This allows us to conclude, by Selberg's lemma, that the image of $\rho$ is finite.

In a forthcoming paper of Biswas--Gupta--Mj--Whang \cite{bgmw}, the methods of Section \ref{sect:2} are also applied in the study of $\SL_2(\C)$-representations of surface groups with finite or bounded mapping class group orbits in the character variety.

\subsection{Deformation techniques}
Our main goal beginning in Section \ref{sect:3} is to prove following result, used in the proof of Theorem \ref{curve} for $g=0$.
\begin{Theorem}\label{genuszero} Suppose $C$ is a  genus $0$ curve with $n$ generic punctures, and let $(V,\nabla)$ be a rank 2 vector bundle with connection on $C$ with $p$-curvatures vanishing for almost all prime numbers $p$. 
If $P \subset C(\C)^{an}$ is any pair of pants, then $(V,\nabla)^{an}$ restricted to $P$ has finite monodromy.
\end{Theorem}

We can explain the significance of Theorem \ref{genuszero} in two ways. First, the content of the $p$-curvature conjecture is to produce a purely topological outcome from arithmetic input. Since a pair of pants is topologically equivalent to $\bP^{1}\setminus{\{0,1,\infty\}}$, one might attempt to show that for a vector bundle with connection $(V,\nabla)$ with vanishing $p$-curvatures on an {\sl arbitrary} curve $X$, its monodromy when restricted to a pair of pants $P \subset X$ is finite. This is precisely what we do, with the assumption that the curve is \emph{generic}.   Secondly, it is well-known that the universal truth of Conjecture \ref{ConjA} for arbitrary $X$ is equivalent to the truth of the conjecture for the particular case $X = \bP^{1} \setminus{\{0,1,\infty\}}$.  However, this equivalence is not rank-preserving.  In this context, our Theorem \ref{genuszero} gives a ``rank-preserving'' way to transfer information from $\bP^{1}\setminus \{0,1,\infty\}$ to generically punctured genus zero curves.

The central idea in the proof is to use nodal degenerations of the generic curve to reducible curves containing $\mathbb{P}^1 \setminus{\{0,1,\infty\}}$ as irreducible components. The proof is carried out in two broad steps. 
\begin{enumerate}
    \item We first show that the vector bundle with connection extends to the irreducible component equal to $\mathbb{P}^1 \setminus{\{0,1,\infty\}}$.
    
    \item We then use that the $p$-curvature conjecture for $\mathbb{P}^1 \setminus{\{0,1,\infty\}}$ is known in rank 2 (this is due to Katz \cite{katz2}), to deduce the finiteness of monodromy of $(V,\nabla)$ restricted to the pair of pants.
\end{enumerate}
The key technique used in carrying out these steps involves a systematic study of the effect of vanishing $p$-curvatures in families of varieties. Our first result along these lines is: 

\begin{Theorem}\label{goodredcurve}
Let $C \rightarrow B$ be a smooth, faithfully flat family of irreducible curves over a smooth irreducible base $B$ defined over a number field. Let $b\in B$ be a codimension-one point whose closure $E$ is defined over a number field, and let $B^o$ be the complement of $E$ in $B$. If $(V,\nabla)$ is a vector bundle on $C \times_B B^o$ with connection relative to $B^o$ with $p$-curvatures vanishing for almost all prime numbers $p$, then $(V,\nabla)$ extends to a vector bundle with connection over the fiber $C_b$. 
\end{Theorem}

This result is one of the main steps involved in the proof of Theorem \ref{genuszero}. It implies (using the Lefschetz hyperplane theorem) a similar result for families of higher-dimensional varieties. To simplify exposition, we  restrict ourselves to the following setting.
\subsubsection{Setting}\label{parasetting}
Let $B$ denote a smooth irreducible variety over a number field $k$, let $X \rightarrow B$ be a smooth projective morphism, and let $(V,\nabla)$ be a vector bundle on $X \times_B B^o$ with connection relative to $B^o$, where $B^o \subset B$ is some nonempty open subvariety defined over $\bar{k}$. %This data can be spread out to $\mathcal{O}_k[1/N]$ for some large enough integer $N \in \Z$. In what follows, we will use the phrase ``vanishing $p$-curvature for almost all primes $p$'' to denote that the $p$-curvature of $(V/p,\nabla/p)$ on $X/p$

\begin{Corollary}[Good reduction]\label{goodred}
Maintain the notation above and let $b \in B$ denote a codimension-one point of $B \setminus B^{o}$. If the $p$-curvatures of $(V,\nabla)$ vanish for almost all primes $p$, then there exists an \'etale neighbourhood $U$ of $b$ such that the pair $(V,\nabla)$ extends to an analytic vector bundle on $X \times_B U$ with connection relative to $U$. 
\end{Corollary}
Corollary \ref{goodred} applies in greater generality than stated. For a more general form of Corollary \ref{goodred}, see Section 3. We note that without the assumption on $p$-curvatures vanishing, there do exist families of vector bundles with connections which do not extend to all codimension-one points of the base. 
%This result has the following consequence. Restrict  $X/B$ to a holomorphic neighbourhood $U^{\hol}$ of the point $b \in B$. By identifying the fundamental groups of the fibers with $pi_1(X_b)$, the family of vector bundles with connection induces a map $\pi_1(X_b) \rightarrow \GL_n(\mathcal{O}^{\hol}(U^{\hol})[1/q])$, where $q$ is a local equation for the codimension-one point $b$. Our result implies that the map has no poles in $q$, i.e. that the map is valued in $\GL_n(\mathcal{O}^{\hol}(U^{\hol}))$.

Our subsequent results concern isomonodromy (see Definition \ref{isomonodromydef} in \S 4 for a precise definition). In general it is expected (and indeed, would follow from Conjecture \ref{ConjA}) that, given a family of vector bundles with flat connection, the vanishing of $p$-curvatures for almost all $p$ implies constancy of the monodromy representation. Let us consider the following conjecture, which is implied by the $p$-curvature conjecture: 
\begin{conjecture}\label{ConjB}
Let $(V,\nabla)$ be a vector bundle with connection on $X$ so that almost all $p$-curvatures vanish. Then, $(V,\nabla)$ has semisimple monodromy. 
\end{conjecture}
Proving Conjecture \ref{ConjA} in the case of semisimple monodromy would imply Conjecture \ref{ConjA} in full generality as the $p$-curvature conjecture is known in the case of solvable monodromy. We prove the following result in \S \ref{sect:4}, conditional on Conjecture \ref{ConjB}.  

\begin{Theorem}\label{isomonodromy}
Let $X \rightarrow B$ denote a family of smooth projective varieties, and let $(V,\nabla)$ denote a vector bundle on $X$ with flat connection relative to $B$. If the $p$-curvatures of $(V,\nabla)$ vanish for almost all primes $p$, then Conjecture \ref{ConjB} implies that the monodromy of $(V,\nabla)$ pulled back to $X_b$ does not depend on $b \in B$. 
\end{Theorem}
 
The proof of this result relies on relating the infinitesimal deformations of $(V,\nabla)_b$ to self-extensions of $(V,\nabla)_{b}$. In the case when the monodromy of some single fiber is finite, then it is possitble to unconditionally show that such non-trivial self-extensions cannot exist. Indeed, the connected component of the Zariski-closure of monodromy would then be non-trivial and unipotent, and the work of \cite{Andre,Bost,Chudnovskis} implies that this is not possible. Therefore, without assuming Conjecture \ref{ConjB}, we get:

\begin{Theorem}\label{oneall}
Maintain the setting of $X \to B$, $(V,\nabla)$ as in \ref{parasetting}, and suppose that the $p$-curvatures vanish for almost all prime numbers $p$. If $(V,\nabla)_{b_0}$ has finite monodromy for some $b_0 \in B$, then $(V,\nabla)_b$ has finite monodromy for every $b\in B$. 
\end{Theorem}
One can compare this result to a theorem of Andre in \cite{Andre}, where he proves the existence of a full set of algebraic solutions to $\nabla$, assuming the finiteness of monodromy of $(V,\nabla)$ restricted to $X_b$ for a Zariski-dense set of $b 
\in B$. Our result is a generalization of Andre's theorem, where we require just one of the fibers to have finite monodromy. However, we remark that our result needs the vanishing of the $p$-curvatures, whereas Andre's doesn't. 

\subsection{Some applications}

Corollary \ref{goodred} and \ref{oneall} can be used in conjunction to deduce new cases of the $p$-curvature conjecture. As a first application, we get:

\begin{Theorem}\label{new}
Let $X_0$ denote a smooth projective variety over some field which is finitely generated over $\Q$ for which the $p$-curvature conjecture is known. Then, the $p$-curvature conjecture is true for the generic fiber of any smooth projective family of varieties which contains $X_0$ as a fiber. 
\end{Theorem}
Tang, in \cite[Theorem 6.1.5]{Yunqing}, proves the $p$-curvature conjecture for the Elliptic curve $E_{1728}\setminus{0}$, under the assumptions that \emph{all} $p$-curvatures vanish. This result suggests that the $p$-curvature conjecture is not equally difficult for all curves of a fixed topological type. In light of this, the significance of Theorem \ref{new} becomes apparent -- one may choose to prove the conjecture for a particularly tractable fiber and thereby deduce it for the generic fiber. Another example that illustrates this principle arises from the work of Farb-Kisin. They \cite{fk} prove Conjecture \ref{ConjA} for certain Shimura varieties (see \cite{fk} to see the exact list which they treat). When the Shimura variety is compact, has dimension $\geq 3$ and is of Hodge type, their proof goes through verbatim for hyperplane slices. Theorem \ref{new} would then apply to the generic fiber of any smooth family of varieties containing one such hyperplane slice as a fiber.

As a second application, we use Theorem \ref{curve} to deduce the $p$-curvature conjecture for rank $2$ connections on certain varieties defined over number fields. For instance, let $X = \A^2 \setminus Y$ where $Y$ is an affine quartic curve with 3 cusps. Then the $p$-curvature conjecture for rank 2 bundles on $X$ follows from Theorem \ref{curve}. For a proof and for further examples, see Section 6.

\subsection{Acknowledgments} 
We express great thanks to Mark Kisin, Andy Putman, Peter Sarnak, Yunqing Tang, and Xinwen Zhu for useful discussions and comments. We thank Ofer Gabber very much for providing examples of representations of fundamental groups which were relevant to our study. We are very thankful to H\'el\`ene Esnault for pointing out an error in an unpublished prior work of ours, and for providing helpful comments to improve the exposition of this paper.

\section{Rank two local systems on surfaces}
\label{sect:2}

The primary goal of this section is to prove Theorems \ref{top} and \ref{top0}, which will be used in the proof of Theorem \ref{curve}. We begin by establishing, in Sections \ref{sect:2.1}  and \ref{sect:2.2}, nonarchimedean and archimedean boundedness properties (corresponding to points (1) and (2) provided by the summary in Section \ref{sect:1.2}) for certain representations of surface groups into $\SL_2(\C)$. In Section \ref{sect:2.3}, we use these results to prove Theorems \ref{top} and \ref{top0}, as outlined in Section \ref{sect:1.2}. Finally, we deduce Theorem \ref{curve}, admitting in advance Theorem \ref{genuszero} (proved in later sections) in the genus zero case.

\subsection{Nonarchimedean bound}
\label{sect:2.1}
The following observation is elementary.

\begin{lemma}\label{nonarch}
Let $\Sigma$ be a topological surface of genus $g\geq0$ with $n\geq0$ punctures. If $\rho:\pi_1\Sigma\to\SL_2(\C)$ is a semisimple representation with finite monodromy along every simple loop on $\Sigma$, then $\tr\rho(a)$ is an algebraic integer for every $a\in\pi_1\Sigma$.
\end{lemma}

\begin{proof}
Fix a base point on $\Sigma$, and let $a:S^1\to\Sigma$ be a based loop. Up to homotopy, we may assume $a$ is an immersion with a minimum possible number (denoted $m(a)$) of self-intersection points in its image, each intersection point being required to be a simple normal crossing. If $m(a)=0$, i.e.~$a$ is a simple loop, then $\tr\rho(a)=\zeta+\zeta^{-1}$ for some root of unity $\zeta$ by our hypothesis on $\rho$, so $\rho(a)$ is an algebraic integer.

So suppose $m(a)\geq1$. We shall proceed by induction on $m(a)$ as follows. Up to homotopy, we may assume that one of the $m(a)$ self-intersection points of $a$ is the base point. The loop $a$ is thus a concatenation $a=bc$ of two uniquely determined nontrivial based loops $b$ and $c$ (namely, travelling along $a$, we define $b$ as the loop of first return to the base point, and $c$ is the remainder). We observe that $m(b)$, $m(c)$, and $m(bc^{-1})$ are all strictly less than $m(a)$. Thus, $\tr\rho(b)$, $\tr\rho(c)$, and $\tr\rho(bc^{-1})$ are all algebraic integers by our inductive hypothesis. But
$$\tr\rho(a)=\tr\rho(bc)=\tr\rho(b)\cdot\tr\rho(c)-\tr\rho(bc^{-1})$$
since $\tr(xy)+\tr(xy^{-1})=\tr(x)\tr(y)$ for every $x,y\in\SL_2(\C)$. This shows that $\tr\rho(a)$ is also an algebraic integer, completing the inductive step and the proof. 
\end{proof}

\subsection{Archimedean bound} \label{sect:2.2}
Let $\Sigma$ be a surface with finitely many punctures. Our goal in this subsection is to prove Lemmas \ref{arch} (for positive genus) and \ref{arch0} (for genus zero) below, which provide sufficient conditions for a representation $\pi_1\Sigma\to\SL_2(\C)$ to be unitarizble (i.e.~have image conjugate to a subgroup of $\SU(2)$), in terms of simple loops or pairs of pants on $\Sigma$.

For the benefit of the reader, we briefly summarize our approach. The essence is to show that, if a semisimple representation $\rho:\pi_1\Sigma\to\SL_2(\C)$ becomes unitarizable upon restriction to every one-holed torus (if $\Sigma$ has positive genus) or every pair of pants (if $\Sigma$ has genus zero) on $\Sigma$, then $\rho$ is unitarizable. The hypothesis implies that $\rho$ has real character, and therefore has image conjugate to a subgroup of $\SU(2)$ or of $\SL_2(\R)$ (see \cite{ms}); thus, it remains to show that in the latter case the image of $\rho$ is conjugate to a subgroup of $\SO(2)$. For this, it suffices to show that the isometric action of $\pi_1\Sigma$ on the hyperbolic plane induced by $\rho$ has a fixed point. We achieve this by considering points fixed by various subsurfaces of $\Sigma$. Below, we recall some elementary facts about $\SL_2(\C)$, and reproduce some topological notions considered in \cite{whang} to conveniently keep track of subsurfaces of $\Sigma$.

\begin{definition}
An element $g\in\SL_2(\C)$ is said to be:
\begin{itemize}
    \item \emph{parabolic} if $\tr(g)=\pm2$ and moreover \emph{central} if $g=\pm1$;
    \item \emph{elliptic} if $\tr(g)\in(-2,2)$; and
    \item \emph{loxodromic} if $\tr(g)\notin[-2,2]$.
\end{itemize}    
\end{definition}
The above classification arises from well-known study of the transitive isometric action of $\SL_2(\C)$ on the hyperbolic three-space. This action restricts to an isometric action of $\SL_2(\R)$ on a hyperbolic plane, which can be identified with the usual action of $\SL_2(\R)$ on the Poincar\'e upper half plane $$\Hb^2=\{z\in\C:\textup{Im}\,z>0\}$$
by M\"obius transformations. The stabilizer of $i\in\Hb^2$ is the special orthogonal group $\SO(2)$. A non-central $g\in\SL_2(\R)$ is elliptic precisely when it is conjugate in $\SL_2(\R)$ to an element of $\SO(2)$, and precisely when it has a unique fixed point in $\Hb^2$.

Let $\Sigma$ be a surface of genus $g$ with $n$ punctures. We fix a base point in $\Sigma$. Recall a standard presentation of the fundamental group
$$\pi_1\Sigma=\langle a_1,d_1,\cdots,a_g,d_g,c_1,\cdots,c_n|[a_1,d_1]\cdots[a_g,d_g]c_1\cdots c_n\rangle.$$
We can choose (the based loops representing) the generators so that the sequence of loops  $(a_1,d_1,\cdots,a_g,d_g,c_1,\cdots,c_n)$ have the property that each loop is simple and any two distinct loops intersect exactly once and at the base point. For $i=1,\cdots,g$, let $b_i$ be the based simple loop parametrizing the curve underlying $d_i$ with the opposite orientation. Note that $(a_1,b_1,\cdots,a_g,b_g,c_1,\cdots,c_n)$ satisfies the following:
\begin{enumerate}
	\item[(1)]each loop in the sequence is simple,	
	\item[(2)]any two distinct loops intersect exactly once (at the base point), and
	\item[(3)]every product of distinct elements in the sequence preserving the cyclic ordering can be represented by a simple loop in $\Sigma$.
\end{enumerate}
Some examples of products alluded to in (3) are $a_1b_g$, $a_1a_2b_2b_g$, and $b_gc_na_1$.

\begin{definition}
An \emph{optimal sequence of generators} for $\pi_1\Sigma$ is a sequence of loops $(a_1,b_1,\cdots,a_g,b_g,c_1,\cdots,c_n)$ obtained from a standard presentation of $\pi_1\Sigma$ as above.
\end{definition}

See Figure \ref{fig1} for an illustration of optimal generators for $(g,n)=(2,1)$.

\begin{figure}[ht]
    \centering
    \includegraphics{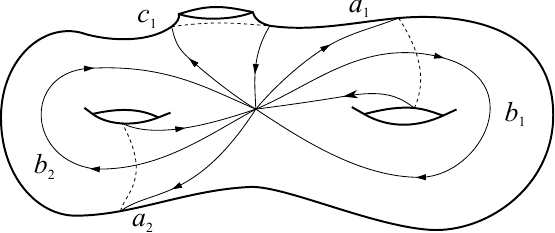}
    \caption{Optimal generators for $(g,n)=(2,1)$}
    \label{fig1}
\end{figure}

\begin{definition}
A pair of loops $(\ell_1,\ell_2)$ on $\Sigma$ is in $(1,1)$-\emph{position} (resp.~\emph{$(0,3)$-position}) if it is homotopic termwise to a pair $(\ell_1',\ell_2')$ of simple loops intersecting exactly once at the base point, with the following property: there is a closed tubular neighborhood $\Sigma'\subset\Sigma$ of the union of images of $\ell_1'$ and $\ell_2'$ in $\Sigma$ which is a subsurface of genus one with one boundary curve (resp.~of genus zero with three boundary curves), and $(\ell_1',\ell_2')$ completes to an optimal sequence of generators for $\pi_1\Sigma'$.
\end{definition}

 For example, suppose $S=(a_1,b_1,\cdots,a_g,b_g,c_1,\cdots,c_n)$ is an optimal sequence of generators for $\pi_1\Sigma$. We note the following.
 \begin{itemize}
     \item If $g\geq1$, then $(a_1,b_1\ell)$ is in $(1,1)$-position for any loop $\ell\neq a_1,b_1$ in $S$.
     \item If $g=0$, then $(c_1,c_i)$ is in $(0,3)$-position for any $i=2,\cdots,n$.
 \end{itemize}
The above definitions and observations facilitate the proofs of our lemmas below.

\begin{lemma}\label{arch}
Let $\Sigma$ be a surface of positive genus $g\geq1$ with $n\geq0$ punctures. If $\rho:\pi_1\Sigma\to\SL_2(\C)$ is a semisimple representation with elliptic or central monodromy along every simple loop of $\Sigma$, then the $\rho$ is conjugate to a representation into $\SU(2)$.
\end{lemma}

\begin{proof}
In the case $(g,n)=(1,1)$, this result can be seen from the characterization of the locus of unitary characters in the $\SL_2(\C)$-character variety of a once-punctured torus (see e.g.~\cite[Section 5]{goldman}). More precisely, if $(a_1,b_1,c_1)$ is an optimal sequence of generators for $\pi_1\Sigma$, and $\rho:\pi_1\Sigma\to\SL_2(\C)$ is a semisimple representation with
$$\tr\rho(a_1),\tr\rho(b_1),\tr\rho(a_1b_1),\tr\rho(c_1)\in[-2,2],$$
then the character of $\rho$ is the character of a representation $\pi_1\Sigma\to\SU(2)$, and the claim follows. We also remark that, if $\Sigma$ is a once-punctured torus, then any representation $\rho:\pi_1\Sigma\to\SL_2(\C)$ with elliptic or central monodromy along every simple loop must automatically be semisimple.

In general, let $\rho:\pi_1\Sigma\to\SL_2(\C)$ be given as above. Since $\tr\rho(a)\in[-2,2]\subset\R$ for each simple loop $a$ on $\Sigma$, it follows by arguing as in the proof of Lemma \ref{nonarch} that $\tr\rho(a)\in\R$ for every $a\in\pi_1\Sigma$. Since $\rho$ is semisimple, this implies that the image of $\rho$ is conjugate to a subgroup of $\SU(2)$ or to a subgroup of $\SL_2(\R)$ (see e.g.~\cite[Proposition III.1.1.]{ms}). If the former occurs, then we are done. So let us assume that $\rho$ has image in $\SL_2(\R)$.

The representation $\rho$ induces an action of $\pi_1\Sigma$ on $\Hb^2$, and it suffices for us to show that this action has a fixed point. By case $(g,n)=(1,1)$, we know that any pair of loops $(\ell_1,\ell_2)$ in $(1,1)$-position on $\Sigma$ must have a common fixed point on $\Hb^2$. Now, we may assume there is a simple loop $a_1$ on $\Sigma$ such that:
\begin{itemize}
    \item the image of $a_1$ has connected complement in $\Sigma$, and
    \item $\rho(a_1)$ is non-central (and therefore elliptic).
\end{itemize}
Indeed, otherwise, the image of $\rho$ can be easily seen to be contained in $\{\pm1\}$, and we are done. Given such a loop $a_1$, let us complete it to an optimal sequence $S=(a_1,b_1,\cdots,a_g,b_g,c_1,\cdots,c_n)$ of generators for $\pi_1\Sigma$. Let $q\in\Hb^2$ denote the unique fixed point of $a_1$. Since the pair $(a_1,b_1)$ is in $(1,1)$-position, it follows that $a_1$ and $b_1$ have a common fixed point in $\Hb^2$, namely $q$. For any other loop $\ell\neq a_1,b_1$ in $S$, the pair $(a_1,b_1\ell)$ is in $(1,1)$-position; thus, $\ell$ also fixes $q$. Since $S$ generates $\pi_1\Sigma$, it follows that $\pi_1\Sigma$ fixes $q$, as desired.
\end{proof}

\begin{lemma}
\label{arch0}
Let $\Sigma$ be a topological surface of genus zero with $n\geq0$ punctures. If $\rho:\pi_1\Sigma\to\SL_2(\C)$ is a semisimple representation whose restriction to any pair of pants is conjugate to a representation into $\SU(2)$, then $\rho$ itself is conjugate to a representation into $\SU(2)$.
\end{lemma}

\begin{proof}
Let $\rho:\pi_1\Sigma\to\SL_2(\C)$ be given as above. Since $\tr\rho(a)\in[-2,2]\subset\R$ for each simple loop $a$ on $\Sigma$, it follows by arguing as in the proof of Lemma \ref{nonarch} that $\tr\rho(a)\in\R$ for every $a\in\pi_1\Sigma$. Since $\rho$ is semisimple, this implies that the image of $\rho$ is conjugate to a subgroup of $\SU(2)$ or to a subgroup of $\SL_2(\R)$ (see e.g.~\cite[Proposition III.1.1.]{ms}). If the former occurs, then we are done. So let us assume that $\rho$ has image in $\SL_2(\R)$.

The representation $\rho$ induces an action of $\pi_1\Sigma$ on $\Hb^2$, and it suffices to show that this action has a fixed point. Let $(c_1,\cdots,c_n)$ be an optimal sequence of generators for $\pi_1\Sigma$. Note that $\rho(c_i)$ is elliptic or central for each $i=1,\cdots,n$ by our hypothesis on $\rho$. We may assume without loss of generality that $\rho(c_1)$ is not central, and let $q\in\Hb^2$ be the unique fixed point of $\rho(c_1)$. For each $i=2,\cdots,n$, the pair $(c_1,c_i)$ is in $(0,3)$-position. The restriction of $\rho$ to $\Sigma_i$ has image in a compact subgroup of $\SL_2(\R)$ by our hypotheses, and hence must have a (unique) fixed point, namely $q\in\Hb^2$. It follows that $c_i$ fixes $q$ for every $i=1,\cdots,n$, and hence $\pi_1\Sigma$ fixes $q$.
\end{proof}

\subsection{Proofs of Theorems \ref{curve}, \ref{top}, and \ref{top0}}\label{sect:2.3}
With the results of Sections \ref{sect:2.1} and \ref{sect:2.2} in hand, we now restate and prove Theorems \ref{top} and \ref{top0}, and deduce from them Theorem \ref{curve} (with the assumption of Theorem \ref{genuszero} in the case $g=0$).

\begin{thm}{\ref{top}}
Let $\Sigma$ be a topological surface of genus $g\geq1$ with $n\geq0$ punctures. If a semisimple representation $\rho:\pi_1\Sigma\to\GL_2(\C)$ of its fundamental group has finite monodromy along every simple loop on $\Sigma$, then the image of $\rho$ is finite.
\end{thm}

\begin{thm}{\ref{top0}}
Let $\Sigma$ be a topological surface of genus zero with $n\geq0$ punctures. If a semisimple representation $\rho:\pi_1\Sigma\to\GL_2(\C)$ has finite restriction to every pair of pants on $\Sigma$, then the image of $\rho$ is finite.
\end{thm}

\begin{proof}[Proof of Theorems \ref{top} and \ref{top0}]
Suppose $\rho:\pi_1\Sigma\to\GL_2(\C)$ satisfies the conditions of Theorem \ref{top} or \ref{top0}. Since the morphism
\begin{equation*}\label{eqeq}
({\textup{pr}},{\det}):\GL_2(\C)\to\PGL_2(\C)\times\GL_1(\C)
\end{equation*}
has finite kernel, it suffices to show that the compositions ${\textup{pr}}\circ\rho:\pi_1\Sigma\to\PGL_2(\C)$ and ${\det}\circ\rho:\pi_1\Sigma\to\GL_1(\C)$ each have finite image. The fact that ${\det}\circ\rho$ has finite image is clear. To show that ${\textup{pr}}\circ\rho$ has finite image, let us first introduce a punctured surface $\Sigma'=\Sigma\setminus p$ for some $p\in\Sigma$. Since the morphism $i:\pi_1\Sigma'\to\pi_1\Sigma$ is surjective, it suffices to show that ${\textup{pr}}\circ\rho\circ i$ has finite image; choosing a semisimple lift $\tilde\rho:\pi_1\Sigma'\to\SL_2(\C)$ of ${\textup{pr}}\circ\rho\circ i$ (which exists since $\pi_1\Sigma'$ is free), it suffices to show that $\tilde\rho$ has finite image. Note that $\tilde\rho$ has finite monodromy along every simple loop on $\Sigma'$.

Thus, it only remains to prove the theorem in the case where the image of $\rho$ lies in $\SL_2(\C)$. We proceed as outlined in Section \ref{sect:1.2}. By Lemma \ref{nonarch}, the character $\tr\rho$ of $\rho$ takes values in the ring of algebraic integers (and \emph{a fortiori} in $\bar\Q$). Since $\rho$ is semisimple, up to $\SL_2(\C)$-conjugation we may thus assume that the image of $\rho$ lies in $\SL_2(\bar\Q)$. Moreover, any conjugate $\rho^\sigma$ of $\rho$ by an element $\sigma\in\Gal(\bar\Q/\Q)$ of the absoluate Galois group of $\Q$, given as
$$\rho^\sigma(a)=(\rho(a))^\sigma\quad\text{for all $a\in\pi_1\Sigma$},$$
satisfies the hypotheses of Theorem \ref{top} or \ref{top0}, since having finite order is preserved under $\Gal(\bar\Q/\Q)$. It follows by Lemma \ref{nonarch} together with Lemma \ref{arch} (for surfaces of positive genus) or Lemma \ref{arch0} (for surfaces of genus zero) that the eigenvalues of $\rho(a)$ for each $a\in\pi_1\Sigma$ are algebraic integers (being roots of the characteristic polynomial $X^2-(\tr\rho(a))X+1=0$) whose Galois conjugates all have absolute value $1$ in $\C$, and by Kronecker's theorem they must be roots of unity, i.e.~$\rho(a)$ has finite order (note that $\rho(a)$ is semisimple by Lemma \ref{arch} or \ref{arch0}). Finally, Selberg's lemma states that any finitely generated subgroup of $\GL_r(L)$ for a field $L$ of characteristic zero has a torsion-free finite-index subgroup. Applied to the image of $\rho$, this shows that the image of $\rho$ is therefore finite.
\end{proof}

\begin{remark}
As communicated by H\'el\`ene Esnault, one could replace the last steps of the above proof by a simpler argument using discreteness and compactness. On the other hand, the proof presented above has some robust features that may be more widely applicable.
\end{remark}

\begin{thm}{\ref{curve}}
The $p$-curvature conjecture is true for rank $2$ vector bundles with connection on a generic curve of genus $g\geq0$ with $n\geq0$ punctures.
\end{thm}

\begin{proof}
Let $(V,\nabla)$ be a rank two vector bundle with connection on a generic curve of genus $g\geq0$ with $n\geq0$ punctures, such that almost all $p$-curvatures of $(V,\nabla)$ vanish. By \cite[Theorem 1.3]{ananth}, every simple loop has finite monodromy. Moreover, in the case $g=0$, the restriction of its monodromy representation (upon specialization) to any pair of pants is finite by Theorem \ref{genuszero}.

Choose a specialization of the generic curve to a curve over $\C$ with associated Riemann surface $\Sigma$, and consider the monodromy representation $\rho:\pi_1\Sigma\to\GL_2(\C)$ associated to $(V,\nabla)$. Since the $p$-curvature conjecture is known to be true in the case of solvable monodromy (by Andre \cite{Andre}, Bost \cite{Bost}, and  D.~Chudnovsky--G.~Chudnovsky \cite{Chudnovskis}), we may assume that $\rho$ is irreducible (and in particular semisimple). It follows by Theorems \ref{top} and \ref{top0} that the image of $\rho$ is finite, as desired.
\end{proof}

\section{Good reduction at the special fiber}
\label{sect:3}

In this section, we first prove Theorem \ref{goodredcurve} and then use it to deduce Corollary \ref{goodred}. 
\subsection{Proof of Theorem \ref{goodredcurve}}
 
%Notation as in the statement of Theorem \ref{goodredcurve}. By assumption, all the data is defined over $k$, a number field. 
%Up to possibly replacing $k$ with a finite extension, there exist $B' \subset B$, an open subvariety which contains $b$, and $C' \subset C|_{B'}$ an open subvariety which contains the  generic point of the fiber $C_b$, such the the following properties hold:

%\begin{enumerate}\label{ass}

%\item  $B' = \Spec R$, $E \subset B'$ is defined by the vanishing of a single equation $q = 0$. 

%\item The vector bundle $V|_{C'|_{B' \cap B^o}}$ is trivial, and there exists a cyclic basis {\bf e} with respect to the endomorphism
%$\nabla(D)$ (see for eg. \cite[Theorem 4.4.2]{kedlaya})\footnote{The reference proves the existence of a cyclic vector when restricted to the generic point of $C$. This can then be spread out to some open subvariety of $C|_{B^o}$, which dictates our choice of $C'$ and $B'$.}. 

%\item The data is defined over $\Oo_k[1/N]$ for a sufficiently large integer $N\in \Z$, $C' = \Spec S$ is affine, $\Omega_{C'/R}$ is the trivial line bundle, and there exists $D \in \Der(\mathcal{O}_C'/R)$ such that $D^p \equiv D \mod \mathfrak{p}$, where $\mathfrak{p}$ is a maximal ideal of $\mathcal{O}_k[1/N]$ with residue characteristic $p$. 

%\end{enumerate}
%The following result shows that it suffices to treat the case of $C'\rightarrow B'$ instead of $C \rightarrow B$.
%\begin{lemma}
%Notation as above. Suppose that $(V,\nabla)$ extends to $C'_b$. Then, $(V,\nabla)$ extends to $C_b$.
%\end{lemma}

%Alternative formulation: 
Maintain the notation from the statement of Theorem \ref{goodredcurve}. Suppose that $B'\subset B$ is an open subvariety containing $b$, the generic point of $E$, and $C'\subset C|_{B'}$ is also open and contains the generic point of $C_b$. The following lemma shows that it suffices to treat the case of $C'\rightarrow B'$, instead of $C\rightarrow B$: 

\begin{lemma}\label{lemma:C'C}
Notation as above. Suppose that $(V,\nabla)$ extends to $C'_b$. Then, $(V,\nabla)$ extends to $C_b$.
\end{lemma}

\begin{proof}
By the condition that $C'$ contain the generic point of $C_b$, the pair $(V,\nabla)$ is defined on the complement of a codimension $2$ closed subset $Z \subset C$.  Since $C/\Spec k$ is smooth, the vector bundle $V$ uniquely extends to a vector bundle $V'$ on the complement of a codimension 3 closed set $W \subset C$ contained in $Z$ (apply  \cite[Corollary 1.2, Corollary 1.4]{hart80} to $(i_{*}V)^{\vee \vee}$, where $i: C\setminus Z \hookrightarrow C$ is inclusion). In particular, $V$ must extend to $C_b$. Finally, by Hartog's theorem, the connection $\nabla$ extends to a connection on $V'$. The lemma follows.
\end{proof}

We now choose $C'$ and $B'$ so as to satisfy the following conditions:
\begin{enumerate}\label{assone}

\item  $B' = \Spec R$ is affine and $E' = E \cap B'$ is defined by the vanishing of a single element $q \in R$. 

\item We have $C' = \Spec S$ is affine, $\Omega_{C'/R}$ is the trivial line bundle, and there exists $D \in \Der(\mathcal{O}_{C'}/R)$ such that for all but finitely many maximal ideals  $\mathfrak{p} \subset \Oo_k$ , we have $D^p \equiv D \mod \mathfrak{p}$ where $\Oo_k/\mathfrak{p}$ has characteristic $p$. 

\item There exists a trivializing basis {\bf e} of sections of $V$ on the open set $C'|_{B' \setminus E'} = \Spec S[1/q]$, with respect to which the endomorphism
$\nabla(D)$ is cyclic (see for eg. \cite[Theorem 4.4.2]{kedlaya})\footnote{The reference proves the existence of a cyclic vector when restricted to the generic point of $C$. This can then be spread out to some open subvariety of $C|_{B^o}$, which dictates our choice of $C'$ and $B'$.}. 

\end{enumerate}

In light of Lemma \ref{lemma:C'C}, we can assume that $C =C'$ and $B=B'$.  We make this assumption from here on.

%A simple argument (which we leave to the reader) using Hartog's theorem shows that vector bundles with connection can be extended across smooth codimension two points. Therefore, in order to prove Theorem \ref{goodredcurve}, it suffices to treat the case of  any $C'\rightarrow B'$ which is obtained as above. 

%By choosing $B'$ to be a small enough open subscheme of $B$ and $C'$ a suitable open subscheme of $C$, and up to replacing $N$ by some larger integer, we may assume the following:
%A simple argument (which we leave to the reader) using Hartog's theorem for extending vector bundles with connection $(V,\nabla)$ across codimension two points allows us to reduce Theorem \ref{goodredcurve} to the case of families $C \to B$ obeying conditions $(1),(2),$ and $(3)$. 

\begin{remark}
  We can extend the vector bundle $V$ (which lives on $C|_{B \setminus E}$) trivially to all of $C$ using the cyclic basis ${\mathbf{e}}$. Then, in order to prove Theorem \ref{goodredcurve}, it suffices to show that the connection matrix of $\nabla$ with respect to $\mathbf{e}$ has no poles in $q$.
\end{remark}

\begin{remark}\label{notrestrictive}
   If $f:X \to Y$ is any family of reduced curves parametrized by a pointed base $(Y,y)$, and if $Z$ is any irreducible component of the fiber $X_y$, then there always exists an open affine subscheme $X' \subset X$ which intersects $Z$ nontrivially and a derivation $D \in \Der({\mathcal{O}_{X'/Y}})$ satisfying condition (2) above. This can be obtained by choosing an appropriate dominant map $X \to \A^{1} \times Y$ and pulling back the derivation $x \cdot d/dx$ -- we leave the details to the reader.
\end{remark}

\begin{definition}
Let $F(S)$ denote the fraction field of $S$, and let $D \in \Der_R(F(S))$. Let $\nu$ denote a discrete valuation on $F(S)$. We say that the derivation $D'$ is $\nu$-{\bf integral} if $\nu(D(\alpha)) \geq \nu(\alpha)$ for all $\alpha \in F(S)$. 
\end{definition}
The fact that $D \in \Der(S/R)$ implies that $D$ is $q$-integral (here, we abuse notation to allow $q$ to denote the $q$-adic valuation on $\Oo_C$).
\begin{theorem}\label{thesis}
Let $C = \Spec S \rightarrow \Spec R$ be the affine curve satisfying the conditions  above, and let $(V,\nabla)$ denote a vector bundle on $C$ with connection relative to $R[1/q]$ satisfying the above assumptions. If the $p$-curvatures of $(V,\nabla)$ vanish for almost all primes $\fp$, then $\nabla$ extends to a connection over $C \rightarrow \Spec R$. 
\end{theorem}
We will spend the rest of this section proving preparatory lemmas and proposition, ultimately proving Theorem \ref{thesis}.
\subsection*{Setup}
The proof of Theorem \ref{thesis} goes along the lines of the arguments used in \cite[Lemma 3.3, Proposition 3.4]{ananth}. We first set up our notation. Let $r$ be the rank of $V$. The connection matrix $A$ (of $\nabla(D)$) with respect to the cyclic basis $\be$ has the form 
\begin{equation*}\label{connectionmatrix}
A=
\left(
\begin{array}{cccc}
&&& f_0\\
1& &&f_1\\
&\ddots & &\vdots\\
&&1&f_{r-1}
\end{array}
\right)
\end{equation*}
In order to prove Theorem \ref{thesis}, it suffices to prove that the $f_i \in S[1/q]$ all have non-negative $q$-adic valuation. Let $\fp \subset \Oo_k$ (with associated rational prime $p$) be such that: (1) $p > r$, (2) the $q$-adic valuations of the $f_i$ stay the same modulo $\fp$, and (3) the $p$-curvature of $(V,\nabla)$ vanishes modulo $\fp$. It then suffices to prove that the reductions of the $f_i$ modulo such primes $\fp$ have non-negative $q$-adic valuation. 

To that end, we work modulo $\fp$ for the rest of the proof, and we assume for sake of contradiction that some $f_i \in S[1/q]$ has negative $q$-adic valuation. We localize $R$ and $S$ at $(q)$, and replace these rings by their $q$-adic completions. Then, the map $C \rightarrow B$ is replaced by $\Spec \kappa_C[[q]] \rightarrow \Spec \kappa_B[[q]]$, where $\kappa_B \subset \kappa_C$ are fields of characteristic $p$. The derivation $D$ naturally extends to a $q$-integral derivation $D \in \Der(\kappa_C[[q]] / \kappa_B[[q]])$. And finally, the connection $\nabla$ yields a map $\nabla(D): \kappa_C((q))^r \rightarrow \kappa_C((q))^r$, with $\nabla(D) (v) = Av + D(v)$, where $A$ is as in Equation \eqref{connectionmatrix}. \autoref{thesis} immediately follows from the following proposition: 

\begin{proposition}\label{reduceto}
Notation as above. Suppose that one of the entries of $A$ has negative $q$-adic valuation. Then, $\nabla(D)^p \neq \nabla(D^p)$, i.e. the $p$-curvature of $\nabla$ doesn't vanish. 
\end{proposition}
The idea of the proof of Proposition \ref{reduceto} is to argue that $\nabla(D)^p(w) \neq \nabla(D^p)(w)$ for an appropriately chosen eigenvector $w$ of $A$, by showing that the $q$-adic absolute value of a component of the vector $\nabla(D)^p(w)$ is strictly greater than its counterpart in $\nabla(D^p)(w)$. The proof of Proposition \ref{reduceto} will be provided after two lemmas.

\subsection*{Notation} To proceed with further, we will need some notation. Let $L / \kappa_C((q))$ denote the (separable, because $p>r$) field extension of $\kappa_C ((q))$ obtained by adjoining the eigenvalues of $A$. The $q$-adic valuation extends uniquely to a valuation $\nu$ on $L$. Let $|\cdot|$ denote the absolute value on $L$ induced by the valuation $\nu$.  Let $\lambda \in L$ denote an eigenvalue of $A$ with \emph{maximal} $\nu$-adic absolute value (i.e. minimal $q$-adic valuation), and let $\ell = |\lambda|$. Let $w_{\lambda}$ denote an eigenvector with eigenvalue $\lambda$. Given any vector $w$, we let $w[m-1]$ denote the $m^{th}$ coordinate of $w$ with respect to some basis which is clear in context.

\begin{lemma}\label{extder}
The derivation $D$ can be $\nu$-integrally  extended to $L$.
\end{lemma}
\begin{proof}
There is a unique extension of $D$ to $L$ (which we will also denote by $D$), and it remains to prove the $\nu$-integrality of $D$. Note that the extension $L / \kappa_C((q))$ cannot have any wild ramification, as $L$ is the splitting field of a degree $r$ polynomial over $\kappa_C((q))$ and $p > r$ by assumption. As the extension $L/\kappa_C((q))$ is tamely ramified, it follows that $L \subset T[q^{1/m}]$, where $T/\kappa_C((q))$ is unramified, and $m \geq 1$.

We claim that extension of $D$ to $T$ is $q$-integral because $T/\kappa_C((q))$ is unramified : indeed, let $\alpha \in \mathcal{O}_T^{\times}$. It suffices to show that the $q$-adic valuation of $D(\alpha) \geq 0$. Let $g \in \kappa_C[[q]][x]$ denote the monic irreducible polynomial which $\alpha$ satisfies. Denote by $D(g)$ the polynomial obtained by applying $D$ to the coefficients of $g$. Then, $D(\alpha)$ is easily seen to equal $\frac{D(g)(\alpha)}{g'(\alpha)}$. As $T /S$ is unramified, it follows that $g'(\alpha)$ is a $q$-adic unit. It is also apparent that the $q$-adic valuation of $D(g)(\alpha)$ is non-negative, as required. 

We now show that $D$ can be $\nu$-integrally extended from $T$ to $L$. Let $s = \displaystyle{\sum_{i = 0}^{m-1} t_i q^{i/m}}$ be an element of $S'$, where $t_i \in T$. The unique extension of $D$ to $S$ satisfies $D(q^{i/m}) = 0$ for every $i$, therefore $D(s) = \sum_i D(t_i)q^{i/m}$. The $q$-adic valuation of $t_i q^{i/m}$ is different from the valuation of $t_j q^{j/m}$ for $i \neq j$. Therefore, $v_q(s) = \min_i \{v_q(t_iq^{i/m}) \}$. For any $i$, the $q$-integrality of $D$ on $T$ yields that $v_q(D(t_i)q^{i/m}) \geq v_q(t_i q^{i/m})$. Putting this together, we see that $v_q(D(s)) \geq \min_i\{v_q(D(t_i)q^{i/m}) \} \geq \min_i \{v_q(t_iq^{i/m}) \} = v_q(s)$ and the lemma follows. 

\end{proof}

\begin{lemma}\label{sofya}
Suppose that a vector $w$ has the property that for all $m$, $|w[m]| \le$ $(resp.$ $<)$ $\ell^{r-m-1} |w_{\lambda}[r-1]|$. Then: 
\begin{enumerate}
\item $|Dw [m]| < \ell^{r-m-1} |(\lambda w_{\lambda})[r-1]|$, i.e. the coordinates of $Dw$ and $\lambda w_{\lambda}$ will satisfy the \emph{strict} inequalities. 

\item For all $m$, $|(A w)[m]| \leq (resp. <) \, \ell^{r-m-1}|\lambda w_{\lambda}[r-1]|$.

\item The vector $w_{\lambda}$ satisfies the property $|w_{\lambda}[m]| \leq \ell^{r-m-1}|w_{\lambda}[r-1]|$.
\end{enumerate}
\end{lemma}
\begin{proof}
Part (1) follows directly from the $\nu$-integrality of $D$.

Since the $f_{m}$ are elementary symmetric expressions in the eigenvalues of $A$ and since $\lambda$ is an eigenvalue with largest absolute value, it follows that  $|f_m| \leq \ell^{r-m}$ holds, with equality holding for some $m$. %Part (1) of the lemma now follows from the $\nu$-integrality of $D$ (Lemma \ref{extder}).

To prove part (2), we must show that $|Aw[m]| \le$ $(resp.$ $<)$ $ \ell^{r-m}|w_{\lambda}[r-1]|$. This is clear for $Aw[0]$ by the assumption on $w$ and the fact that $Aw[0] = f_0 w[r-1]$. For $m \geq 1$, we have $Aw[m] = f_mw[r-1] + w[m-1]$. Since $|f_m| \leq \ell^{r-m}$, it follows that $|f_mw[r-1]| \le$ $(<)$ $|f_m w_{\lambda}[r-1]| \le \ell^{r-m}|w_{\lambda}[r-1]|$. Therefore it suffices to show that $|w[m-1]| \le$ $(<)$ $\ell^{r-m}|w_{\lambda}[r-1]|$. But this holds by hypothesis.

Finally, we prove the third part by induction on $m$. As in the proof of part (2), we get $\lambda w_{\lambda}[0] = (A \cdot w_{\lambda})[0] = f_0w_{\lambda}[r-1]$. Thus, the claim is proven for $m = 0$. For $m >0$, we have  $\lambda w_{\lambda}[m] = A w_{\lambda}[m] = w_{\lambda}[m-1] + f_mw_{\lambda}[r-1]$. It follows that $\ell|w_{\lambda}[m]| \leq \max \{|w_{\lambda}[m-1]|, |f_m w_{\lambda}[r-1]|\}$. The term  $|w_{\lambda}[m-1]|$ is bounded by $\ell^{r-m} |w_{\lambda}[r-1]|$ by induction, and $|f_m w_{\lambda}[r-1]|$ is bounded by $\ell^{r-m} |w_{\lambda}[r-1]|$ because $|f_m| \leq \ell^{r-m}$ for all $m$. The lemma follows.

\end{proof}

We are now ready to prove Proposition \ref{reduceto}. 
\begin{proof}[Proof of Proposition \ref{reduceto}]
We will prove that $\nabla(D)^p w_{\lambda} \neq \nabla(D^p) w_{\lambda}$, from which the result obviously follows. We now introduce some notation. For any vector $v$, denote by $v[m]$ its $(m+1)^{th}$ entry. 
We have
$$\nabla(D)^p w_{\lambda} = \sum_{W \in \it{I}} W w_{\lambda},$$
where $\it{I}$ is the set of all length $p$ words in the letters $A$ and $D$. Here, $A$ acts on any vector by left-multiplication, and $D$ acts on each coordinate in the natural way.

As in \cite{ananth}, we will show using Lemma \ref{sofya} that the word $W_0 = AA\hdots A$ ($p$ times) has the property that $|(W_0w_{\lambda})[r-1]|$ is strictly larger than $|(Ww_{\lambda})[r-1]|$ for every $W_0 \neq W \in \it{I}$. 

 Lemma \ref{sofya}(3) implies that $|w_{\lambda}[m]| \leq \ell^{r-m-1}|w_{\lambda}[r-1]|$. Let  $w^j$ and $w_0^j$ be the vectors obtained by applying the first $j$ letters of $W$ and $W_0$ respectively on $w_{\lambda}$. Lemma \ref{sofya} implies $|w^j[m]| \le \ell^{r-m-1}|w_0^j[m]|$ for any $0\le m \le r-1$. However, $W$ differing from $W_0$, must contain the letter $D$. Suppose that  such a letter first occurred at the $j_0^{th}$ stage. Then, Lemma \ref{sofya} (1) implies that $|w^{j_0}[m]| < \ell^{r-m-1}|w_0^{j_0}[r-1]|$. According to Lemma \ref{sofya} (2), $|w^{j}[m]| < \ell^{r-m-1}|w_0^{j}[r-1]|$, for $j \ge j_0$. 
Therefore, $|(W_0w_{\lambda})[r-1]| > |(Ww_{\lambda})[r-1]|$ for every $W \neq W_0$, and hence $|\nabla(D)^p(w_{\lambda})[r-1]|= |(A^p w_{\lambda}) [r-1]| = \ell^p|w_{\lambda}[r-1]|$. On the other hand,  by Lemma \ref{sofya} (1) and the hypothesis that $D^p = D$, $$|\nabla(D^p)(w_\lambda)[r-1]| = |(D+A)(w_\lambda)[r-1]| =  \ell|w_{\lambda}[r-1]| .$$ Thus $\nabla(D)^p w_{\lambda} \neq \nabla(D^p) w_{\lambda}$, which is what we needed to show.
\end{proof}

\begin{proof}[Proof of \autoref{thesis}]
\autoref{thesis} immediately follows from Proposition \ref{reduceto}.
\end{proof}

\subsection{Families of holomorphic vector bundles with connection}\label{familyhol}

%-----------------------------------Anand Insert--------
Here we spell out for the reader standard constructions and results involving families of vector bundles with connection. For a reference, see \cite{del}.

Our usual setting is: $f: X \to B$ is a smooth morphism between smooth complex manifolds with $d$-dimensional fibers, and $p: B \to X$ is a section.  Further, $(V, \nabla)$ is a rank $r$ holomorphic vector bundle with integrable connection relative to $f$ on $X$. 

When the family $f$ is locally (on $B$) homeomorphic to a product, we review the construction of the corresponding holomorphic family of monodromy representations attached to $(V,\nabla)$.  We  assume that the reader is familiar with the basic correspondence between locally constant sheaves on a manifold and representations of its fundamental group.

The essential ingredient is:

\begin{lemma}[Cauchy-Kowalewski Theorem]
    The sheaf of solutions $\ker \nabla$ is locally (in the Euclidean topology on $X$) isomorphic to the inverse-image sheaf
    $f^{-1}\Oo_{B}^{r}$. 
    \label{Cauchy-Kowalewski}
\end{lemma}
\begin{proof}
    This statement is contained in the proof of \cite[Theorem 2.23]{del}. It is known as the Cauchy-Kowalewski theorem
    on existence and uniqueness of solutions to certain types of differential
    equations.
\end{proof}

\begin{corollary}
    Let $i : X_{b} \hookrightarrow X$ denote the inclusion of the fiber
    over $b \in B$. Then
    the sheaf $i^{-1}\ker \nabla$ is a locally constant sheaf of
    free rank $r$ $\Oo_{B,b}$-modules on $X_{b}$ (in the Euclidean topology). 
\end{corollary}

\subsubsection{Relative local systems and monodromy representations}

Throughout this section, let $f: X \to B$ be a smooth morphism between complex manifolds, which is locally on $B$ homeomorphic to a product. Furthermore, assume $p: B \to X$ is a section, and suppose the fundamental groups of the fibers of the family are finitely generated.  All sheaves considered here are in the analytic topology.

\begin{definition}
	A {\bf $f$-local system} of rank $r$ on $X$ is a sheaf of $f^{-1}\Oo_{B}$-modules on $X$ which is locally (in the Euclidean topology) isomorphic to $(f^{-1}\Oo_{B})^{r}$.

	If $M$ is an $f$-local system, we let $p^{*}M$ denote the rank $r$ locally free sheaf of $\Oo_{B}$-modules $p^{-1}(M)$.
\end{definition}

We define $\underline{\pi_{1}(X/B)}$ to be the locally constant sheaf on $B$ whose stalk at a point $b \in B$ is $\pi_{1}(X_{b}, p(b))$. 

\begin{proposition}\label{prop:ForwardCorrespondence}
	A $f$-local system $M$ on $X$  naturally defines a homomorphism of sheaves of groups: $\rho_{M}: \underline{\pi_{1}(X/B)} \to \Aut(p^{*}M)$.
\end{proposition}

\begin{proof}
	If $b \in B$ is any point and $i_{b}: X_{b} \hookrightarrow X$ is the inclusion of the fiber, we obtain a locally constant sheaf $i^{-1}(M)$ of $\mathcal{O}_{B,b}$-modules on $X_{b}$. The correspondence between locally constant sheaves and representations of $\pi_{1}(X_{b},p(b))$ on the stalk $(p^{*}(M))_{b}$  provides a canonical  representation
	\begin{align*}
		\rho_{b}: \pi_{1}(X_{b},p(b)) \to \Aut(p^{*}M)_{b}
	\end{align*}
	Since the above fundamental group is finitely generated, there exists an open neighborhood $U \subset B$ of $b$ and a representation $\rho_{U}: \pi_{1}(X_{b},p(b)) \to \Aut(p^{*}M)(U)$ whose germ is $\rho_{b}$.  

	The representations $\rho_{U}$ are clearly compatible on intersections, and collectively define $\rho_{M}$.
\end{proof}

If $M$ is a local system relative to $f$, we call $\rho_{M}$ the {\bf monodromy representation of $M$}. 

\begin{proposition}\label{prop:BackwardCorrespondence}
    Let $W$ be a vector bundle on $B$, and $\rho: \underline{\pi_{1}(X/B)} \to \Aut(W)$ a homomorphism of sheaves of groups. Then there exists a unique $f$-local system $M$ such that $\rho_{M} = \rho$.
\end{proposition}

\begin{proof}
    The $f$-local system $M$ is constructed as follows.  Let $\{U_{\alpha}\}$ be a cover of $B$ consisting of contractible open sets and such that every pairwise and triple intersection are also contractible, and let $\pi_{\alpha}: \tilde{X}_{\alpha} \to X_{\alpha} \to U_{\alpha}$ denote the universal covers of $X_{\alpha}$ mapping down to $U_{\alpha}$.  Then the sheaf $\pi^{-1}_{\alpha}(W)$ inherits, through $\rho_{U_{\alpha}}$, an action of $\pi_{1}(X_{\alpha},p_{\alpha})$ which is compatible with the action of $\pi_{1}(X_{\alpha},p_{\alpha})$ on the universal cover $\tilde{X}_{\alpha}$. Thus, the sheaf descends to a sheaf $M_{\alpha}$ on $X_{\alpha}$, which is locally on $X_{\alpha}$ isomorphic to $f^{-1}(W)$ by construction. 

    Next, over an intersection $U_{\alpha \beta} = U_{\alpha} \cap U_{\beta}$ we get natural sheaf isomorphisms $\phi_{\alpha \beta}: M_{\alpha}|_{X_{\alpha \beta}} \to M_{\beta}|_{X_{\alpha \beta}}$ by noticing that the identity map $X_{\alpha \beta} \times \pi_{\alpha}^{-1}(W|_{U_{\alpha}})|_{\alpha \beta} \to X_{\alpha \beta} \times \pi_{\beta}^{-1}(W|_{U_{\beta}})|_{\alpha \beta}$ is equivariant with respect to the action of $\pi_{1}(X_{\alpha}, p_{\alpha})$ on the left and $\pi_{1}(X_{\beta}, p_{\beta})$ on the right. (Note that the existence of the section $p$ provides  canonical inclusions of the universal cover $\tilde{X}_{\alpha \beta}$ into the universal covers $\tilde{X}_{\alpha}$ and $\tilde{X}_{\beta}$.) The equivariance comes from the agreement of the restrictions of $\rho_{U_{\alpha}}$ and $\rho_{U_{\beta}}$ to $U_{\alpha \beta}$.

    Similarly, on triple intersections we get that the cocycle condition $\phi_{\beta \gamma} \circ \phi_{\alpha \beta} = \phi_{\alpha \gamma}$ holds automatically. 
\end{proof}
To summarize, the content of Propositions \ref{prop:ForwardCorrespondence} and \ref{prop:BackwardCorrespondence} is that the category of $f$-local systems is equivalent to the category of monodromy representations. 

\begin{proposition}\label{proposition:extension}
        Let $i : C \hookrightarrow X$ be a complex submanifold containing the section $p$ which is locally topologically trivial over $B$, with the property that the induced map $\underline{\pi_{1}(C/B)} \to \underline{\pi_{1}(X/B)}$ is surjective.

         If $L$ is a $f$-local system on $C$ which, over a dense open subset $U \subset B$ is the inverse image of an $f$-local system $M_{U}$ on $X_{U}$,  then $M_{U}$ extends to an $f$-local system $M$ on $X$. 
     \end{proposition}   
\begin{proof}
    The $f$-local system $L$ gives, by Proposition \ref{prop:ForwardCorrespondence}, a monodromy representation 
    \begin{align*}
        \rho_{L}: \underline{\pi_{1}(C/B)} \to \Aut(p^{*}{L})
    \end{align*}
    which factors through $\underline{\pi_{1}(X/B)}$ over the dense open subset $U \subset B$ by hypothesis. By continuity, $\rho_{L}$ must factor through $\underline{\pi_{1}(X/B)}$ on all of $B$, and therefore defines (by Proposition  \ref{prop:BackwardCorrespondence}) the required $f$-local system $M$.
\end{proof}

We make the following remark for future use:

\begin{remark}
  \label{rem:nottrivial}
  If $f:X \to B$ is a smooth morphism between complex varieties with
  section $p:B \to X$ and if $b \in B$ is a any point, then the
  discussion in this section attaches to a vector bundle with
  integrable connection $(V,\nabla)$ on $X$ a monodromy representation
  \begin{align}
    \label{eq:monrep}
    \rho: \pi_1(X_b, p(b)) \to \Aut p^{*}(V)_{b}
  \end{align}
  This local monodromy representation exists irrespective of any
  hypotheses on the topological triviality of the family $f$.
\end{remark}

%----------------------------------------------------

%---------------------------------------------------------------

\subsection{Proof of Corollary \ref{goodred}}

\begin{proof}[Proof of Corollary \ref{goodred}]
Let $X \to B$ be a smooth projective family of $k$-varieties of relative dimension $d$, and $D \subset B$ an irreducible divisor. Suppose further that $(V,\nabla)$ is an algebraic vector bundle with relative flat connection on $X|_{B \setminus D}$ whose $p$-curvatures vanish for almost all $p$.

Then because $f$ is smooth, there exists a quasi-finite base change $U \to B$ whose image contains the generic point of $D$ and which is generically unramified over $D$, such that $X \times_{B}U \to U$ has a section $p: U \to X\times_{B}U$. We choose a divisor $D' \subset U$ lying over $D$, and we replace $B$ with $U$ and $D$ with $D'$. $U$ will serve as the \'etale neighborhood of the generic point of $D$ in the statement of the theorem.

Next, by choosing a suitable collection of relatively very ample divisors  $H_{1}, ..., H_{d-1}$ of $X$ containing the section $p$ and after shrinking $U$ if necessary, Bertini's theorem guarantees that the relative curve $C := H_1 \cap \dots H_{d-1}$ is smooth over $U$.  Furthermore, the Lefschetz hyperplane theorem for fundamental groups guarantees that the induced map
\begin{align*}
    \underline{\pi_{1}(C/U)} \to  \underline{\pi_{1}(X/  U)}
\end{align*}

is surjective. 

The vector bundle with relative flat connection $(V,\nabla)$ on $X|_{U\setminus D'}$ restricts to a vector bundle with connection on $C|_{U\setminus D'}$, and after replacing $C$ with a suitable affine open subscheme, we get a relative curve satisfying the hypotheses in Theorem \ref{thesis} -- see Remark \ref{notrestrictive}. The affine open can be chosen to contain the generic point of the section $p$ over $U$ and to guarantee that $C \to U$ is locally homeomorphic to a product.

By Theorem \ref{thesis}, the vector bundle with relative connection $(V,\nabla)$ extends over $C$, and therefore  the hypotheses in Proposition \ref{proposition:extension} are fulfilled.  By that proposition, we conclude that the $f$-local system $\ker \nabla$ extends over $X$, and by the correspondence between $f$-local systems and vector bundles with relative flat connection, we produce the required analytic extension claimed in  Corollary \ref{goodred}. This completes the proof.
 \end{proof}

\begin{remark}
The argument in the proof of Corollary \ref{goodred}  can be adapted to the following slightly more general situation: $X \to B$ is the complement of a relative simple normal crossing divisor in a smooth projective family $Y \to B$.  We leave the details to the reader.
\end{remark}

\section{Constancy of monodromy}
\label{sect:4}
We spend this section proving Theorems \ref{isomonodromy}, \ref{oneall} and \ref{new}. We first need the following result, which proves formal constancy of the monodromy representation in the setting of Theorems \ref{isomonodromy} and \ref{oneall}. 

We first define what we mean by isomonodromy. 

\begin{definition}\label{isomonodromydef}
Let $C \rightarrow B$ denote a homolomorphic family of smooth curves, inducing a holomorphic map $B \rightarrow \Mgn$, and let $(V,\nabla)$ be a holomorphic vector bundle on $C$ with connection relative to $B$.
\begin{enumerate}
\item Given two points $P,Q \in B$, and a real-analytic non self-intersecting path $\alpha$ connecting $P$ and $Q$, there is a canonical way to identify $\pi_1(C_P)$ and $\pi_1(C_Q)$ up to inner-automorphisms. We say that \emph{the monodromy of $(V,\nabla)$ at $P$ is the same with respect to the path $\alpha$ as the monodromy of $(V,\nabla)$ at $Q$} if the monodromy representation $\rho_P: \pi_1(C_P) \rightarrow \GL_r(\C)$ associated to $(V,\nabla)|_{C_P}$ is isomorphic to the monodromy representation $\rho_Q: \pi_1(C_Q)\rightarrow \GL_r(\C)$, under the identifications of $\pi_1(C_P)$ and $\pi_1(C_Q)$ induced by $\alpha$. 

\item We say that \emph{$(V,\nabla)$ is isomonodromic} if given any two points $P,Q \in B$ and given any non-self intersecting real analytic path $\alpha$ connecting $P$ and $Q$, the monodromy of $(V,\nabla)$ at $P$ is the same with respect to the path $\alpha$ as the monodromy of $(V,\nabla)$ at $Q$.
\end{enumerate}
\end{definition}

\begin{definition}\label{constantmodm}
Given a family of smooth affine curves $f: C \to \Spec k[[q]]$ with central fiber $C_0$, and a vector bundle with connection $(V,\nabla)$ relative to $B$, we say that the connection $\nabla$ is {\sl constant modulo $q^m$} if $(V,\nabla)|_{C \times_{k[[q]]} \Spec k[[q]]/(q^m)} $ is isomorphic to $(V,\nabla)|_{C_0} \times \Spec k[[q]]/(q^m)$. Here we are implicitly using the infinitesimal lifting property to identify $C \times _{k[[q]]} \Spec k[[q]]/(q^{m})$ with $C_{0} \times _{k} \Spec k[[q]]/(q^{m})$.  
\end{definition}

\begin{remark}\label{constantmodmHol}
Definition \ref{constantmodm} has its natural analogue in the holomorphic setting, where $\Spec k[[q]]$ is replaced by the unit disk $\Delta$, and $k[[q]]$ is replaced by the ring $\C\{q\}$ of holomorphic functions on $\Delta$.  Note that if $\Delta \to \Mgn$ is a holomorphic map inducing a family of curves $C \to D$, if $(V,\nabla)$ is vector bundle with connection on $C$ relative to $\Delta$, and if $z \in \Delta$ is any  point, then by choosing a coordinate $q_{z}$ around $z$, one can replicate Definition \ref{constantmodm} to get the notion of constancy of $(V,\nabla)$ modulo $q_{z}^{m}$.  We will use this holomorphic setting in Proposition \ref{isomon} below.
\end{remark}

\begin{proposition}\label{formalconstancy}
Let $C \rightarrow \Spec k[[q]]$ denote a smooth affine curve, and let $(V,\nabla)$ denote a trivial\footnote{Note that the connection is \emph{not} assumed to be trivial, only the vector bundle is.} vector bundle on $C$ with connection relative to $\Spec k[[q]]$. Suppose that the data of $\{ C,(V,\nabla) \}$ arises upon completion from an algebraic family of curves and an algebraic vector bundle with connection, and that the $p$-curvatures vanish for almost all prime numbers $p$. Further suppose that either 
\begin{enumerate}
    \item There exists a full set of algebraic solutions modulo $q$, or
    \item Conjecture \ref{ConjB} is true.
\end{enumerate}
Then for any positive integer $m$, there exists a basis of sections trivializing $V$ with respect to which the connection matrix of $\nabla$ is constant (in $q$) modulo $q^m$. 
\end{proposition}

\begin{proof}
Without loss of generality, we assume the existence of a derivation $D \in \Der(\Oo_C)$, such that $D^p \equiv D \mod p$.

We proceed by induction on $m$, and will assume the existence of a basis such that the connection matrix $\nabla(D)$ is constant modulo $q^{m-1}$. Let $C_0$ denote the special fiber of $C$. As $C_0$ is affine and smooth, the infinitesimal lifting property implies that $C \times_{k[[q]]} k[[q]]/q^m$ is isomorphic to the trivial deformation $C_0 \times_k k[[q]]/q^m$. Therefore, in terms of this trivialization of $C$ mod $q^m$, the connection matrix $\nabla(D)$ can be expressed as $A + q^{m-1}B$, with $A,B \in M_{n \times n}(\Oo_{C_0})$. 

Consider the rank $2r$ trivial vector bundle $V'$ with connection $\nabla'$ on $C_0$, whose connection matrix with respect to $D$ is 
\[
M =
\left[
\begin{array}{c|c}
A & B \\
\hline
0 & A
\end{array}
\right].
\]

We will prove the following two statements: 

\begin{claim}\label{cone}
If there exists a block upper-triangular change of coordinates of $V'$ (with coefficients in $\Oo_{C_{0}}$) with respect to which $\nabla'(D)$ is block-diagonal, then the assertions of Proposition \ref{formalconstancy} hold. 
\end{claim}

\begin{claim}\label{ctwo}
The $p$-curvatures of $\nabla'$ vanish for almost all primes $p$.  
\end{claim}

The Proposition follows immediately from these claims. Indeed, if $\nabla \mod q$ had a full set of algebraic solutions, then the connected component of the monodromy representation associated to $\nabla'$ would be solvable. We further know by Claim \ref{ctwo} that the $p$-curvatures of $(V',\nabla')$ vanish. As the $p$-curvature conjecture is known when the connected component of monodromy is solvable, we obtain that $(V',\nabla')$ has a full set of algebraic solutions, and therefore that there exists a block upper-triangular change of coordinates with respect to which $\nabla'(D)$ is block-diagonal. Proposition \ref{formalconstancy} then follows by applying Claim \ref{cone}. 

On the other hand, regardless of the monodromy of the special fiber, Conjecture \ref{ConjB} and Claim \ref{ctwo} imply that $(V',\nabla')$ has semisimple monodromy. Thus there exists a block upper-triangular change of coordinates with respect to which $\nabla'(D)$ is block diagonal, and so the proposition would follow from the assertion of Claim \ref{cone}. Therefore, it suffices to prove the two claims. 

\begin{proof}[Proof of Claim \ref{cone}]
Suppose that the block upper-triangular matrix is of the form 
\[
G =
\left[
\begin{array}{c|c}
X & Y \\
\hline
0 & X'
\end{array}
\right].
\]
By again changing coordinates by the matrix 
\[
\left[
\begin{array}{c|c}
X^{-1} & 0 \\
\hline
0 & X'^{-1}
\end{array}
\right],
\]
it follows that we may assume the initial block upper-triangular matrix $G$ has its block-diagonal entries equalling the identity. 

Further, the connection matrix $\nabla' (D)$ undergoes the gauge transformation $G^{-1} M G + G^{-1} D(G)$. A short calculation shows that the top-right block of $\nabla'(D)$ becomes $B + AY - YA + D(Y)$. Therefore, the assumption in Claim \ref{cone} means that there exist a matrix $Y$ such that $B + AY -YA + D(Y) = 0$. 

We now shift our focus back to $\nabla$. Consider the gauge transformation induced by the change of basis matrix $I + q^{m-1} Y$. A short calculation shows that $\nabla(D)$ in these new coordinates equals $A - q^{m-1}(B + AY - YA + D(Y))$, and this quantity equals $A$ as our previous calculation yielded that $B + AY - YA + D(Y) = 0$. The claim follows. 
\end{proof}

\begin{proof}[Proof of Claim \ref{ctwo}]
We work modulo $p$ for the entirety of this proof. As $D^p \equiv D \mod p$, it suffices to prove that $\nabla'(D) \equiv \nabla'(D)^p$. Suppose that the matrix $\nabla(D)^j = P_j + q^{m-1}Q_j$, for some positive integer $j$. We will show by induction that that the matrix $\nabla'(D)^j$ equals 
\[
\left[
\begin{array}{c|c}
P_j & Q_j \\
\hline
0 & P_j
\end{array}
\right].
\]
By induction, we may assume that this holds for $j-1$. Then, $\nabla(D)^j = (A + q^{m-1}B)(P_{j-1} + q^{m-1}Q_{j-1}) + D(P_{j-1} + q^{m-1}Q_{j-1}) = AP_{j-1} + D(P_{j-1}) + q^{m-1}(AQ_{j-1} + BP_{j-1} + D(Q_{j-1}))$. Therefore, $P_j = AP_{j-1} + D(P_{j-1}) $ and $Q_j = AQ_{j-1} + BP_{j-1} + D(Q_{j-1})$. 

On the other hand, a similar calculation shows that 

\[
\nabla'(D)^j = 
\left[
\begin{array}{c|c}
AP_{j-1} + D(P_{j-1}) & AQ_{j-1} + BP_{j-1} + D(Q_{j-1}) \\
\hline
0 & AP_{j-1} + D(P_{j-1})
\end{array}
\right].
\]
The claim now follows from the vanishing of the $p$-curvatures of $\nabla$. 
\end{proof}
Now that both the claims have been proved, the Proposition follows.
\end{proof}
The following result is the other ingredient needed to prove Theorems \ref{isomonodromy} and \ref{oneall}.

\begin{proposition}\label{isomon}
Suppose $C \rightarrow \Delta$ is a family of affine curves over the unit disc $\Delta$ with coordinate $q$ which induces a holomorphic map $\Delta \rightarrow \mathcal{M}_{g,n}$, and suppose $(V,\nabla)$ is a holomorphic vector bundle on $C$ with connection relative to $\C\{q\} = \mathcal{O}^{\hol}(\Delta)$.  For each point $z \in \Delta$ and integer $m$, suppose that  $\nabla$ is constant modulo $q_z^m$, where $q_{z}$ is a local coordinate at $z$. Then $(V,\nabla)$ is isomonodromic. 
\end{proposition}

\begin{proof} As the base $\Delta$ is simply connected, there is a canonical identification of the fundamental groups of all fibers of $C$ -- we denote this group by $\pi_1(C)$. After choosing a section, the family of vector bundles with connection gives rise to a representation $\rho: \pi_1(C) \rightarrow \GL_r(\C\{q\})$ (see \S \ref{familyhol}). Let $R$ denote any $\C\{q\}$-algebra. We denote by $\rho_R$ the induced representation of $\pi_1(C)$ valued in $\GL_r(R)$. For any point $z$, we let $\rho_z:\pi_1(C) \rightarrow \GL_r(\C)$ denote $\rho$ specialized to $q = z$. 

Let $L$ denote the field of fractions of $\C\{q\}$. We will show that $\rho_0$ and $\rho_L$ are conjugate under $\GL_{r}(\overline{L})$. This suffices to prove the claim, because the same would hold for $\rho_z$ and $\rho_L$, for any point $z$. It follows that the representations $\rho_0$ and $\rho_z$ are isomorphic over $\overline{L}$, and thus isomorphic over any algebraically closed field, in particular $\C$.

Therefore, it suffices to prove that $\rho_0$ and $\rho_{\overline{L}}$ are isomorphic. We first prove the following claim: 
\begin{claim}
Let $\sigma: \pi_1(C) \rightarrow \GL_r(\C)$ and $\tau:  \pi_1(C)\rightarrow \GL_r(\C[q]/(q^{m+1}))$ be representations such that $\tau \equiv \sigma \mod q^m$ and $\tau$ is isomorphic to $\sigma$. Then, there exists a matrix $M \in M_r(\C)$ such that $(I + q^m M)^{-1} \tau (I + q^{m} M)= \sigma$.
\end{claim}
\begin{proof}[Proof of claim]
By assumption, there exists a matrix $A \in \GL_r(\C[q]/q^{m+1})$ such that $A^{-1} \tau A = \sigma$. Write $A = B + q^m B_m$ with $B_m$ constant and such that $B$ doesn't have any terms involving $q^{m}$. As $\sigma \equiv \tau$ modulo $q^{m}$, it follows that $B$ commutes with $\sigma$. Therefore, conjugating $\tau$ by $A$ is the same as conjugating $\tau$ by $AB^{-1}$. Clearly, $AB^{-1} \equiv I$ mod $q^m$, and so has the form $I + q^mM$. The claim follows.
\end{proof}
Applying this claim inductively yields that $\rho_0$ is isomorphic to $\rho_{\C[[q]]}$. Therefore, $\rho_0$ is isomorphic to $\rho_{\C((q))}$. As $L \subset \C((q))$, it follows that $\rho_0$ and $\rho_L$ are isomorphic over $\C((q))$. Therefore, they are isomorphic over $\overline{L}$, as required.

\end{proof}

We can now provide the proofs of Theorems \ref{isomonodromy} and \ref{oneall}.  The key ingredient in both proofs is Proposition \ref{formalconstancy}.

\begin{proof}[Proof of Theorem \ref{isomonodromy}] 
In order to prove the result it suffices to treat the case when $B$ is a smooth connected curve, so we make this assumption. Further, by passing to hyperplane slices and applying the Lefschetz hyperplane theorem we may assume that $X \rightarrow B$ is a family of smooth proper curves. By removing an appropriate closed subset of $X$, it suffices for us to treat the case when $X\rightarrow B$ is a family of smooth affine curves, which induces a map $B \rightarrow \Mgn$ so as to obey the hypotheses in Propositions \ref{formalconstancy} and \ref{isomon}. 

 Let $b \in B$ denote any point and let $q$  denote a uniformizing parameter at $b$. By applying Proposition \ref{formalconstancy}, we deduce that the connection is constant modulo $q^m$, and that this is true for all $m$ and all $b'\in B$. We now choose a simply-connected analytic neighbourhood $U$ of $b \in B$. By Proposition \ref{isomon}, it follows that the monodromy of $(V,\nabla)$ restricted to any $X_{b'}$ for $b' \in U$ is independent of $b'$. As $U$ was an arbitrary simply connected open subset of $B$, it follows that the monodromy of $(V,\nabla)$ restricted to any $X_{b'}$ for $b' \in B$ is independent of $b'$. The result follows. 

\end{proof}

\begin{proof}[Proof of Theorem \ref{oneall}]
We assume that there exists $b\in B$ such that the monodromy of $(V,\nabla)$ restricted to $X_b$ is finite. Further, we may assume that $B$ is a smooth, connected curve and that $X\rightarrow B$ is a smooth family of affine curves as in the proof of Theorem \ref{isomonodromy}. We may also assume that the monodromy at the fiber $X_b$ is trivial, by passing to an appropriate finite cover of $X$. Let $q$ denote a uniformizing parameter for the point $b \in B$. Choose a simply-connected holomorphic neighbourhood $U \subset B(\C)$ of $b$ biholomorphic to the unit disk $\Delta$, and let $\C \{q\} = \mathcal{O}^{\hol}(U)$. 

Consider the monodromy representation $\rho: \pi_1(X_{b}) \rightarrow \GL_r(\C \{ q\})$ obtained after making standard choices as in the proof of Proposition \ref{isomon}. By Proposition \ref{formalconstancy}, we obtain that $\rho \mod q^m$ is isomorphic to $\rho \mod q$. However, we have that $\rho \mod q$ is the trivial representation (as we've assumed that the monodromy of the fiber $X_b$ is trivial), and so it follows that $\rho \mod q^m$ is trivial for all $m$. Therefore, it follows that $\rho $ itself is trivial, which yields the fact that the monodromy of $(V,\nabla)$ restricted to $X_{b'}$ is trivial for any $b' \in U$. As $U$ was arbitrary, it follows that the monodromy of $(V,\nabla)|_{b'}$ is trivial for every $b' \in B$. The theorem follows. 
\end{proof}

We are finally ready to deduce Theorem \ref{new} from Corollary \ref{goodred} and Theorem \ref{oneall}. 
\begin{proof}[Proof of Theorem \ref{new}]

Let $X \rightarrow B$ denote such a family, and suppose that $(V,\nabla)$ is a vector bundle on $X_{B^o} \rightarrow B^o$ with flat connection relative to $B^o$ with vanishing $p$-curvatures for almost all primes $p$. Without loss of generality, we may assume that $B$ is a curve. Let $b \in B$ denote a point such that $X_b = X_0$. The result follows from Theorem \ref{oneall} if $b \in B^o$ so we assume this is not the case. 

Corollary \ref{goodred} implies that $(V,\nabla)$ extends (analytically) to a vector bundle with connection $(V',\nabla')$ on $X_b$. By \cite{del}, $(V', \nabla')$, which is a-priori an analytic vector bundle with connection, has a canonical algebraic structure. %We now outline the logical structure of our proof. 

%\begin{enumerate}
 %   \item We first show that the $p$-curvatures of $(V',\nabla')$ vanish for almost all primes $p$, and thus $(V',\nabla')$ has finite monodromy. 
  %  \item We next pass to the case when $X$ is a family of curves by considering plane slices of $X$. Now, the family of vector bundles with connection $(V,\nabla)$ extends to an {\sl algebraic} family of vector bundles with connection on $ \rightarrow B$ (using Theorem \ref{goodredcurve}).
    
   % \item We then apply Theorem \ref{oneall} to conclude.
%\end{enumerate}
%We now implement this strategy. 
First we claim:
\begin{claim}
Suppose that $(V',\nabla')$ on $X_b$ has finite monodromy. Then the theorem follows. 
\end{claim}
\begin{proof}[Proof of Claim]
Let $C \rightarrow B$ be a family of curves obtained by slicing $X\rightarrow B$ as in the Lefschetz hyperplane theorem. Theorem \ref{goodredcurve} yields that $(V,\nabla)|_{C\times_B B^\circ}$ extends algebraically to a pair $(V_C,\nabla_C)$ on all of $C\rightarrow B$. By construction, we have $(V_C,\nabla_C)^{\hol}|_{C_b}$ is the same as $(V',\nabla')^{\hol}|_{C_b}$, and so the former also has finite monodromy. The claim now follows from Theorem \ref{oneall}.
\end{proof}

Therefore, it suffices to prove that $(V',\nabla)'$ has finite monodromy. 
By hypothesis, it suffices to prove that the $p$-curvatures of $(V',\nabla')$ vanish for almost all prime numbers. Note that  Theorem \ref{goodredcurve} implies that any subfamily of curves $C \rightarrow B \subset X \rightarrow B$ obtained by taking plane-slices has the property that family $(V_{C},\nabla_{C})$ is algebraic, and hence the $p$-curvatures of the resctriction of $(V',\nabla') $ to $C_b$ vanish for almost all primes. We conclude using the following claim:

\begin{claim}
Suppose that $Y$ is a smooth quasi-projective variety over $\overline{\F}_p$. Let $(V,\nabla)$ denote a vector bundle with flat connection on $Y$, such that the $p$-curvature of $(V,\nabla)$ restricted to $C$ vanishes for every smooth plane-section $C$ of $Y$. Then the $p$-curvature of $(V,\nabla)$ vanishes on all of $Y$.
\end{claim}
\begin{proof}[Proof of Claim]
By induction on the dimension $d$ of $Y$, we may assume that the $p$-curvature of $(V,\nabla)$ pulled back to every smooth hyperplane section of $Y$ vanishes. 

Suppose $Q \in Y$ is a point and $Z \subset Y$ is a smooth divisor containing $Q$. Then we get the conormal exact sequence: 
\begin{align}\label{conormal}
    0 \to \mathcal{I_{Z}}/\mathcal{I_{Z}}^{2} \to \Omega_{Y}|_{Z} \to \Omega_{Z} \to 0
\end{align}

By \cite{katz}, the $p$-curvature $\Psi$ is a section of the coherent sheaf of $\Oo_{Y}$-modules $\Hom(V,V) \otimes (\Omega_{Y})^{(p)}$. Here $^(p)$ denotes the Frobenius twist of $\Omega_{Y}$. Furthermore, its restriction $\Psi_{Z} \in \Hom(V,V)|_{Z} \otimes_{\Oo_{Z}} (\Omega_{Y})^{(p)}|_{Z} $ agrees with the $p$-curvature of the restriction of the pair $(V,\nabla)$ to $Z$.

Now choose a point $Q \in Y$, and suppose $Z_{1}, ..., Z_{d}$ are smooth divisors such that the conormal vectors $\mathcal{I_{Z}}/\mathcal{I}_{Z_{i}}^{2}|_{Q} \in \Omega_{Y}|_{Q}$ form a basis.  Then the natural map of vector spaces 
\begin{align}\label{injective}
   \Omega_{Y}|_{Q} \to \bigoplus_{i=1}^{d}\Omega_{Z_{i}}|_{Q} 
\end{align}
is injective.

By Bertini's theorem, at a general point $Q \in Y$, there exists such a collection of divisors $Z_{i}$.  By induction, we may assume the $p$-curvatures of $(V|_{Z_{i}}, \nabla|_{Z_{i}})$ vanish, and hence by the injectivity of \ref{injective}, we conclude that $\Psi$ also vanishes. This concludes the proof.

%Let $\Psi$ denote the $p$-curvature of $(V,\nabla)$. Let $D \in \Der(\mathcal{O}_Y)$ and consider the map $\Psi(D): V \rightarrow V$ (by \cite[5.0.5]{katz}, this map is $\mathcal{O}_Y$-linear). Hence, the vanishing locus of $\Psi(D)$ is a closed subvariety of $Y$. Therefore, it suffices to prove that for $D_i$ ranging over a $\mathcal{O}_Y$-basis of $\Der(\mathcal{O}_Y)$, $\Psi(D_i)$ vanishes on a Zariski-dense set of $Y$. 

%By the Bertini smoothness theorem, there exists a Zariski-dense set of $\overline{\F}_p$ points of $Y$ such that the generic hyperplane through one of these points intersects $Y$ smoothly. Choose any such point $Q \in Y(\overline{F}_p)$. Asking for whether $\Psi(D)$ vanishes at $Q$ for a fixed $D$ is equivalent to asking whether $\Psi(D)$ vanishes on some hyperplane section $Z$ through $Q$, where $D$ doesn't not vanish identically along $Z$. Bertini's theorem in conjunction with our induction hypothesis guarantees the existence of such a smooth hyperplane section. Therefore it follows that $\Psi(D)$ vanishes at a Zariski-dense set of points in $Y$. The claim follows.
\end{proof}

%As the $p$-curvature conjecture is known for $X_0$ (by hypothesis), it follows that $(V',\nabla')$ has finite monodromy. %In order to deduce that $(V,\nabla)$ has finite monodromy, we replace $B$ with an \'etale neighbourhood $U$ of $b$ and $X\rightarrow B$ with $C \rightarrow U$, such that the map $\pi_1(C_u) \rightarrow \pi_1(X_u)$ is surjective. It suffices to prove that $(V,\nabla)$ pulled back to $C$ has finite monodromy. We know that $(V,\nabla)$ pulled back to $C_b$ has finite monodromy (as we have proved that $(V,\nabla)_b$ has finite monodromy). Hence we may apply Theorem \ref{oneall} to finish the proof of this theorem. 
\end{proof}

\section{Proof of Theorem \ref{genuszero}}
 To aid the reader, we first give a sketch of the argument. There are three main inputs to our proof. We first use the genericity of the punctures to specialize to a nodal curve containing $\mathbb{P}^1 \setminus\{0,1,\infty\}$ as an irreducible component, and use Theorem \ref{goodredcurve} to prove that $(V,\nabla)$ extends to a vector bundle with connection on this irreducible component. By work of Katz \cite{katz2}, the $p$-curvature conjecture is known for rank two bundles\footnote{We thank H\'el\`ene Esnault for pointing out to us that every rank 2 vector bundle with connection on $\mathbb{P}^1 \setminus{0,1,\infty}$ is Hypergeometric.} on $\mathbb{P}^1 \setminus{0,1,\infty}$, and this is our second input. Finally, we use (a slight generalization of) the isomonodromy Theorem \ref{oneall} to deduce our result.

%Steps: 
%\begin{enumerate}
%\item Good reduction gives a family of curves specializing to a torus with one marked hole, along with vector bundle with connection. Analysis of local monodromy gives that the only singular points on the one-hold torus is the marked hole. 
%There's only one thing to check here: that it's possible to find a ``$q$-integral" derivation $D$ with the property that $D^p \equiv D \mod p$. In section 3.1, we had this because we had a family of smooth topologically constant curves, but in this setting we might need to fight a bit more. Here's how we got the derivation $D$ in 3.1: we had a map $C \rightarrow \mathbb{P}^1$, which was finite flat and extended to the special fiber. Think of this on the level of rings. There's a derivation on the coordinate ring of $\mathbb{P}^1$ which did what we wanted, and all we had to do was extend it to a derivation of $\mathcal{O}_C$, which one can do uniquely, because $\mathcal{O}_C$ is a finite $\mathcal{O}_C$-algebra. 

%\item Every simple closed loop in the one-holed torus can be deformed to the smooth locus, and thus has finite monodromy. Use Theorem in section 2 to deduce finite monodromy on the one-holed torus. 

%\item Every one-holed subtorus of generic curve can actually be realised as one-holed torus as in part 1, and then use holomorphicity of the connection to deduce that this one-holed subtorus also has finite monodromy. 

%\end{enumerate}

%----------------------------------------Anand Stuff-------------------

We work over a number field $k$. Let $f: C \to B$ be a flat family of genus $0$, $n$-punctured (marked) curves over a pointed affine curve $(B,0)$ with smooth
generic fiber and with special fiber $C_0 = f^{-1}(0)$ reduced and
nodal, corresponding to a map $B \to \overline{\mathcal{M}}_{0,n}$, the moduli space of stable $n$-pointed genus $0$ curves. Let $B^\circ$ denote the open set $B \setminus \{0\}$, assume that $C$ is smooth over $B^{\circ}$, and suppose that there exists 
 an irreducible component $E$ of the special fiber $C_0$ which contains exactly three special points (either punctures/marked points or nodes). We let $C^\circ \subset C$ denote the preimage of $B^{\circ}$ under $f$.
 
 Upon blowing down the remaining components $C_{0} \setminus E$, and after appropriately choosing three of the $n$-marked points/punctures to be at $0$, $1$, and $\infty$, we get an affine curve $A \subset (\mathbb{P}^{1} \setminus\{0,1,\infty\}) \times B$ which is the complement of $n$ sections $\sigma_{i}: B \to \mathbb{P}^{1}$ satisfying $\sigma_{i}(0) \in \{0,1,\infty\} \subset \mathbb{P}^1$ for all $i$.  (The first three sections are taken to be the constant sections $0,1,\infty$).  By construction, the affine curve $A$ is isomorphic to an open subset of $C$ under the blowdown map.  
 
  This general setup will be referred to consistently throughout the
  remainder of the section.

\begin{lemma}
  \label{lem:derivationexistence}
  Let $A$ be as above. Then there exists an everywhere non-zero $B$-derivation $D$ on $A$ such that $D^{p} \equiv D \mod p$
  for all primes $p$.
\end{lemma}

\begin{proof}
The derivation $D = x \frac{d}{dx}$ on $A \subset \A^{1} \times B$ works, where $x$ denotes the coordinate on $\A^{1}$.
\end{proof}

Suppose $Y \to B$ is a smooth map of complex varieties, and suppose $b \in B$ is a point sufficiently close to $0 \in B$.  Then there is a map $\tau: \pi_{1}(Y_{0}) \to \pi_{1}(Y_{b})$, well-defined up to choosing base-points, obtained by parallel transport of loops.  

In the context of the family $A \to B$, if we bound the three punctures in $E$ with small loops, the resulting pair of pants $U_{0} \subset E$ can be extended (by Ehresmann's theorem) to a smooth family of pairs of pants $U \to N$ over a small analytic disk $N$ around $0$.  We denote by $U_b \subset A_b$ the fiber of $U$ over $b \in B$. We let $P \subset U$ denote the interior. For each $b \in N$, the fiber $P_b \subset U_b$ is the interior of the pair of pants $U_{b}$.

\begin{theorem}[Bootstrapping]
  \label{thm:bootstrap}
 Suppose $(V,\nabla)$ is a vector bundle with flat connection on $A$ relative to $B$ whose $p$-curvatures vanish for almost all $p$.  
  If the restriction $(V,\nabla)|_{E}$ has finite monodromy, then 
  $\tau(\pi_{1}(E)) = \pi_{1}(P_{b}) \subset \pi_{1}(A_{b})$ has finite image in $\GL_{r}\C$
  under the monodromy representation of $(V,\nabla)|_{A_{b}}$ for $b \in B$ sufficiently close to $0$. Here, $P_b$ is as in Figure \ref{fig6}.
\end{theorem}

\begin{proof}
\begin{figure}[ht]
    \centering
    \includegraphics{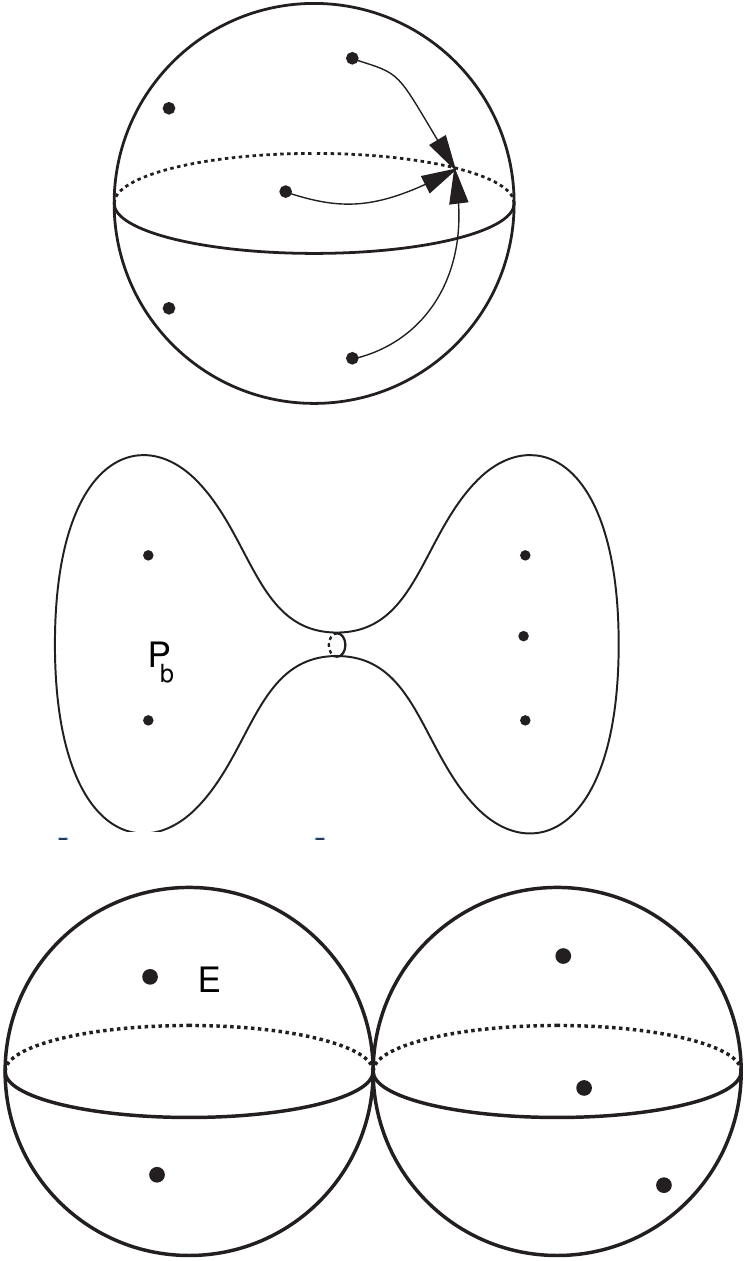}
    \caption{Degeneration of smooth genus zero curves to a nodal curve}
    \label{fig6}
\end{figure}

 %Let $C \rightarrow B$ be a one parameter family of proper curves with the following properties: 
%\begin{enumerate}
%\item The family is generically smooth and has genus $g$.
%\item There exists a smooth point $b\in B(\C)$ and $\gamma \in C_b(\C)$, a separating simple closed loop, and a point $b_o \in B(\C)$  a nodal curve, which can be realised by degenerating $C_b$ by contracting $\gamma$. 

%\item The two irreducible components of $C_{b_0}$ have genera 1 and $g-1$. 
%\end{enumerate}
%Let $C' \subset C$ be a family of smooth affine curves, where $C'_b$ equals the genus one curve with one puncure. Let $(V,\nabla)$ be a vector bundle on $C'$ with connection relative to $B$, such that the $p$-curvatures of $(V,\nabla)$ vanish for almost all $p$, and the monodromy of the fiber over $b_0$ is finite. Then, the monodromy of $(V,\nabla)$ restricted to $U \subset C_b(\C)$ is finite. (Here, $U$ is an open subset of $C_b(\C)$ in the analytic topology. It is defined to be the genus one connected component of $C_b(\C) \setminus \gamma$, which corresponds to the genus one irreducible component of $C_{b_o}$.  

Let $q$ denote a local coordinate on $B = \Spec R$ cutting out $0$, and let $N \subset B(\C)$ denote a small analytic open neighbourhood of $0$ as in the discussion immediately prior to the theorem. Furthermore, let $P \rightarrow N$ denote the holomorphic family of (interiors of) pairs of pants and let $P_b$ be the fiber of $P$ over $b \in N$, as depicted in Figure \ref{fig6}.

It suffices to prove that the monodromy of $(V,\nabla)$ restricted to any fiber of $P \to N$ is finite. The  argument used in the proof of Proposition \ref{formalconstancy} (and also the fact that $A \times_{\Spec R} \Spec R/q^m$ is the trivial infinitesimal deformation of $A \times_{\Spec R} \Spec R/q$) yields that $(V,\nabla)|_{A/q^m}$ is isomorphic to the trivial family $(V,\nabla)_0 \times \Spec R/q^m$. Note that $P$ is an open subset of $A^{\hol}$. Therefore,  there exists a basis for $V^{\hol}|_P$ with respect to which the connection is constant (in $q$) modulo $q^m$. Now, consider the holomorphic family of monodromy representations: $\rho: \pi_1(E) \rightarrow \GL_r(\mathcal{O}^{\hol}(N))$ corresponding to $(V,\nabla)|_{P}$. As the connection is isomorphic to a constant connection mod $q^m$, it follows that the kernel of $\rho$ mod $q^m$ is independent of $m$. We now claim that the kernel of $\rho$ is the same as the kernel of $\rho$ mod $q$. Let $\alpha \in \pi_1(E)$ be in the kernel of $\rho $ mod $q$. Then, $\alpha$ is in the kernel mod $q^m$ for all $m$, and so $\rho(\alpha)$ is the identity element, as required. The result follows, as we have assumed that the kernel of $\rho$ mod $q$ is finite index in $\pi_1(E)$. 
%
%
%Now consider the family of Riemann surfaces $U \subset N$, and the holomorphic vector bundle on $U$ with connection relative to $N$. This is a topologically constant family, and yields a 
%
%The simple closed loop $\gamma$ is well defined (as the vanishing cycle) in every smooth fiber over an appropriate holomorphic neighbourhood $N \subset B(\C)$ of $b_o$. For every $b'$ in this neighbourhood, let $U_{b'}\subset C'_{b'}(\C)$ be the open (in the analytic topology) subset, which corresponds to the genus one connected component of $C_{b'}(\C) \setminus \gamma$. 
\end{proof}

We are now ready to put the above results together to prove Theorem \ref{genuszero}. 
\begin{proof}[Proof of Theorem \ref{genuszero}]
Let $\mathcal{M}_{0,n}$ denote the (fine) moduli space of genus zero curves with $n$ marked points, with $n \geq 4$. We will prove our result for \emph{the} generic family of curves over $\mathcal{M}_{0,n}$. The same arguments in \cite{ananth} go through to reduce the case of \emph{a} generic curve to the case of \emph{the} generic curve. Let $\overline{\mathcal{M}_{0,n}}$ denote the Deligne-Mumford compactification of $\mathcal{M}_{0,n}$, which is a projective variety. There are three different families of ``pairs of pants'' inside $C(\C)$, where $C$ is a genus 0 curve with $n$ punctures: 
\begin{enumerate}
\item Pick any simple closed loop $\gamma \subset C(\C)$ as in Figure \ref{fig6}, such that the complement of $\gamma$ consists of 2 disks, the first containing two of the marked points and the second, containing the remaining $n-2$ marked points. The disc containing two of the marked points is a pair of pants contained inside $C(\C)$, and can be realised by approaching a divisorial boundary component of $\mathcal{M}_{0,n}$. We say this type of pair of pants is of type $P$.\footnote{We thank Joe Harris and Ian Morrison for allowing us to use Figure 2, which can be found in their book {\sl Moduli of Curves}.}
    
\item Fix one of the marked points, and partition the remaining $n-1$ points into two non-empty sets containing $a$ and $b = n-1-a$ points respectively. Let $\gamma_a \subset C(\C)$ denote a simple closed loop bounding the set of $a$ marked points, and let $\gamma_b$ denote the analogous simple closed loop. The complement of $\gamma_a$ and $\gamma_b$ consists of two discs (containing $a$ and $b$ marked points respectively), and a pair of pants. The pair of pants can be realised by approaching a suitable codimension-2 locus in the boundary of $\overline{\mathcal{M}}_{0,n}$. We say such a pair of pants is of type $P_{a,b}$.
    
\item Partition the $n$ points into three nonempty sets containing $a$, $b$ and $c = n -a-b$ points respectively. Let $\gamma_a \subset C(\C)$ denote a simple closed loop which bounds the first $a$ points, and let $\gamma_b,\gamma_c$ denote the analogous simple closed loops. The complement of these three loops equals the union of a pair of pants and three discs (with $a$,$b$ and $c$) punctures respectively. The pair of pants can be realised by approaching a suitable codimension-3 locus in the boundary of $\overline{\mathcal{M}}_{0,n}$. We say such a pair of pants is of type $P_{a,b,c}$.
\end{enumerate}
For more details about $\overline{\mathcal{M}}_{0,n}$, see \cite[Chapter 3, Section G]{HM}. We now blow $\overline{\mathcal{M}}_{0,n}$ up at all the codimension 2 and codimension 3 boundary points considered just above. Let $\overline{\mathcal{M}}$ denote this blown-up scheme; $\overline{\mathcal{M}}$ is still projective, and we fix a projective embedding. 

Let $B$ denote a general one-dimensional curvi-linear section of $\overline{\mathcal{M}}$. We may assume that $B$ is irreducible, is defined over a number field, and also that the map $\pi_1(B \cap \mathcal{M}_{0,n} (\C)) \rightarrow \pi_1(\mathcal{M}_{0,n}(\C)) $ is surjective (by the quasi-projective Lefschetz theorem). We also have that $B$ intersects every divisorial boundary component of $\overline{\mathcal{M}}$. 

Fix any type of pairs of pants $T$, where $T$ either equals $P$,  $P_{a,b}$ (for a fixed pair of integers $a,b$ with $a +b = n-1$), or $P_{a,b,c}$ (with $a+b+c=n$). By \cite[Page 37]{fm}, the action of $\pi_1(M(\C))$ (by parallel transport) on a fiber over $M$ is transitive (up to isotopy) on pairs of pants of type $T$, and thus the same is true about the action of $\pi_1(B \cap M(\C))$. In sum, given any pair of pants $P$ contained in a smooth fiber $C_b$, there exists a path contained in $B(\C)$ connecting $b$ to an appropriate boundary point $0 \in B$ such that the pair of pants deforms to $\mathbb{P}^1 \setminus\{0,1,\infty\} \subset C_0$. We now apply Theorem \ref{thm:bootstrap}, noting that the $p$-curvature conjecture is known (by work of Katz) for rank 2 vector bundles on $\mathbb{P}^1 \setminus\{0,1,\infty\}$. The theorem follows.
\end{proof}

\section{Complements of hypersurfaces in projective space}
\label{sec:complements}

Theorem \ref{curve} can be used to prove more cases of the
$p$-curvature for rank $2$ connections on specific varieties defined over a
number field. We begin with a summary of the basic strategy, and then
provide the new examples. We work over $\overline{\mathbb{Q}}$.  We denote by $[\mathcal{M}_{g,n}/S_n]$ the stack quotient, which parametrizes genus $g$ curves with a marked set of un-ordered $n$ points.

\begin{corollary}
  \label{strategy}
  Suppose $X$ is a smooth variety, $f: C \to B$ is a family of genus
  $g$, $n$-punctured curves, and $p : C \to X$ is a map
  satisfying the following properties:
  \begin{enumerate}
  \item For a general fiber $C_b$, the induced map on fundamental
    groups $$p_{*}:\pi_{1}C_b \to \pi_{1}X$$ is surjective,
  \item The induced map to moduli $B \to [\mathcal{M}_{g,n} / S_{n}]$
    is dominant.
  \end{enumerate}
  Then the $p$-curvature conjecture holds for rank $2$ connections on
  $X$.
\end{corollary}

\begin{proof}
  If $(V, \nabla)$ is a vector bundle with connection on $X$ whose
  $p$-curvatures vanish for almost all $p$, then the pullback
  $(p^{*}V, p^{*}\nabla)$ induces a vector bundle with connection
  relative to $B$ with the same $p$-curvature vanishing condition.
  Since (by condition $(2)$) the map to moduli
  $B \to [\mathcal{M}_{g,n}/ S_{n}]$ is dominant, a general fiber
  $C_b$ is a generic curve, and therefore Theorem \ref{curve} applies.
  Hence, $(p^{*}V,p^{*}\nabla)|_{C_{b}}$ has finite monodromy, which
  (by condition $(1)$) implies that $(V,\nabla)$ also has finite
  monodromy.
\end{proof}

\begin{remark}
The simplest way to produce families $B$ which satisfy condition $(1)$
in Corollary \ref{strategy} is to take a family of ample complete-intersection curves on
$X$. However, such families will typically fail to yield a dominant
map $B \to \mathcal{M}_{g,n}$ once the genus $g$ is large, since the moduli spaces
$\mathcal{M}_{g,n}$ usually are not unirational. Therefore,
the strategy is particularly promising when $g=0$.
\end{remark}

  To obtain a large collection of examples $X$ for which our strategy
  applies, we consider $$X = \mathbb{P}^{r} \setminus Y$$ where
  $Y$ is a degree $n \leq 2r+1$ hypersurface.  The family of curves
  $B$ is the open subset of the Grassmannian of lines in
  $\mathbb{P}^{r}$ which meet $Y$ transversely.  For a general
  hypersurface $Y$ of degree $n \leq 2r+1$, the induced map to moduli
  $B \to \mathcal{M}_{0,n}$ will be dominant, though it is a
  non-trivial problem to determine explicit, independently verifiable
  conditions on $Y$ which guarantee this dominance.
  
  \subsection{A specific class of examples}
  
  In what follows, let $Y \subset \mathbb{P}^{2}$ be a quintic curve.  
  If $Z$ is a degree $d$ plane curve, a {\sl simple tangent} of $Z$ is
  a line which meets $Z$ at precisely $n-1$ smooth points of $Z$.
  
\begin{proposition}
  \label{lemma:dominancecriteria}
  Suppose $Y' \subset \mathbb{P}^{2}$ is an irreducible component of $Y$ such that the family
  of four-pointed lines given by the simple tangent lines $T_p(Y')$ for
  general $p\in Y'$ has varying cross-ratios.  Then the induced map
  to moduli $\mathbb{P}^{2*} \dashrightarrow [\mathcal{M}_{0,5}/S_5]$
  is dominant. In particular, if there exists a line $\ell$ meeting
  $Y$ (set-theoretically) at exactly three points which is a limit of
  simple tangent lines $T_p(Y')$, then the map to moduli is
  dominant.
\end{proposition}

\begin{proof}
  The proof is omitted.
\end{proof}

For a typical example of a line $\ell$ as in the statement of Proposition \ref{lemma:dominancecriteria}, one can take the cuspidal tangent of an ordinary cusp on $Y$, provided it meets $Y$ at two other points.  One can also use a simple bi-tangent line.
  
\begin{corollary}\label{cor:specific}
  Suppose $Y \subset \P^{2}$ obeys the hypothesis in Proposition \ref{lemma:dominancecriteria}.  Then the $p$-curvature conjecture holds for rank $2$ connections on $X = \P^{2} \setminus Y$.
\end{corollary}

\begin{proof}
This is an immediate consequence of Corollary \ref{strategy}.
\end{proof}

\bibliographystyle{alpha}

\end{document}